\documentclass[12pt]{amsart}

\usepackage{amsmath,amssymb,amsthm,multicol,mathtools,hyperref,dsfont,pinlabel,enumitem}
\usepackage[normalem]{ulem}
\usepackage[usenames,dvipsnames]{color}
\usepackage{soul}
\usepackage{tikz}
\usetikzlibrary{matrix}
\usepackage{mathrsfs}
\usepackage{subfiles}
\usepackage[mathscr]{euscript}
\usepackage{comment}

\usepackage{mathdots}
\usepackage{marginnote}
\usepackage{mathabx}
\usepackage{xparse}
\usepackage{marginnote}
\usepackage[foot]{amsaddr}

\makeatletter
\@namedef{subjclassname@2020}{%
  \textup{2020} Mathematics Subject Classification}
\makeatother

\usepackage[T1]{fontenc}

\newtheorem{theorem}{Theorem}[section]

\newtheorem{prop}[theorem]{Proposition}
\newtheorem{cor}[theorem]{Corollary}

\newtheorem{fact}[theorem]{Fact}

\theoremstyle{definition}

\newtheorem{question}[theorem]{Question}

\newtheorem{example}[theorem]{Example}
\newtheorem{rem}[theorem]{Remark}

\newtheorem*{T:MainMatrix}{Theorem \ref{T:MainMatrix}}
\newtheorem*{T:MainKnotLink}{Theorem \ref{T:MainKnotLink}}
\newtheorem*{T:MainQuadratic2}{Theorem \ref{T:MainQuadratic2}}
\newtheorem*{T:Main4Manifold1}{Theorem \ref{T:Main4Manifold1}}
\newtheorem*{T:Main4Manifold2}{Theorem \ref{T:Main4Manifold2}}
\newtheorem*{T:RationalExtension}{Theorem \ref{T:RationalExtension}}
\newtheorem*{C:square}{Corollary \ref{C:square}}

\numberwithin{equation}{section}

\newcommand\lla{\left\langle}
\newcommand\rra{\right\rangle}
\newcommand\lb{\left\bracevert}
\newcommand\rb{\right\bracevert}
\newcommand\cut{\setminus\!\setminus}
\newcommand\wt{\widetilde} 
 
\newcommand\bbm{\begin{bmatrix}}
\newcommand\ebm{\end{bmatrix}}
\newcommand\0{\textbf{0}}
\newcommand\vb{\textbf{v}}
\newcommand\wb{\textbf{w}}
\newcommand\G{\mathcal{G}}
\newcommand\C{\mathbb{C}}
\newcommand\CP{\mathbb{CP}}
\newcommand\Q{\mathbb{Q}}
\newcommand\R{\mathbb{R}}
\newcommand\Z{\mathbb{Z}}

\newcommand\red[1]{\color{red}#1\color{black}}
\newcommand\FG[1]{\color{ForestGreen}#1\color{black}}
\newcommand\violet[1]{\color{Violet}#1\color{black}}
\newcommand\Navy[1]{\color{NavyBlue}#1\color{black}}
\newcommand\white[1]{\color{white}#1\color{black}}
\newcommand\brown[1]{\color{Brown}#1\color{black}}
\newcommand\Yel[1]{\color{Yellow}#1\color{black}}
\newcommand\Orange[1]{\color{BurntOrange}#1\color{black}}
\definecolor{lllg}{gray}{0.9}
\definecolor{llg}{gray}{0.8}
\definecolor{lg}{gray}{0.7}
\definecolor{mlg}{gray}{0.6}
\definecolor{mg}{gray}{0.5}
\definecolor{mdg}{gray}{0.4}
\definecolor{dg}{gray}{0.3}
\definecolor{ddg}{gray}{0.2}
\definecolor{dddg}{gray}{0.1}

\newcommand\Lg[1]{\color{lg}#1\color{black}}
\newcommand\mlg[1]{\color{mlg}#1\color{black}}

\newcommand{\bs}[1]{\left[#1\right]}
\newcommand{\bbme}[1]{\begin{bmatrix}#1\end{bmatrix}}
\newcommand{\bme}[1]{\begin{matrix}#1\end{matrix}}
\newcommand{\ar}{\xrightarrow[]{}}
\newcommand{\ps}[1]{\left(#1\right)}
\newcommand{\zeros}{\mathbf 0}
\newcommand{\bbmes}[1]{\left[\begin{smallmatrix}#1\end{smallmatrix}\right]}

\begin{document}

\title[Kink-equivalence]{Kink-equivalence of matrices, spanning surfaces, 4-manifolds, and quadratic forms}


\author[H. Howards]{Hugh Howards$^1$}
\address{$^1$Department of Mathematics\\
Wake Forest University\\
Winston-Salem, NC 27109}
\email{howards@wfu.edu}

\author{Thomas Kindred$^2$}
\address{$^2$Department of Mathematical Sciences\\
Smith College\\
Northampton, MA 01063}
\email{tkindred@smith.edu}

\author{W. Frank Moore$^1$}
\email{moorewf@wfu.edu}

\author{John Tolbert$^3$}
\address{$^3$Pfafftown, NC 27040}
\email{johntolbert617@gmail.com}

\date{}

\begin{abstract}
All checkerboard surfaces for a given knot in $S^3$ are related by isotopy and (geometric) ``kinking" and ``unkinking" moves, which change matrices representing the surfaces' Gordon-Litherland pairings like this: $G\leftrightarrow G\oplus [\pm1]=$\scalebox{.6}{$ \bbm G&\0\\ \0^T&\pm1 \ebm$}. We call two symmetric integer matrices ``kink-equivalent" if they are related by (matrix) ``kinking'' and ``unkinking'' moves $G\leftrightarrow G\oplus [\pm1]$ and unimodular congruence.  We prove constructively that every nonsingular symmetric integer matrix is kink-equivalent to a positive-definite matrix and to a negative-definite matrix, and we give bounds on the number of moves required. This implies that every knot in $S^3$ is ``alternating up to fake unkinking moves." It also has implications for 4-manifolds and quadratic forms.
\end{abstract}




\maketitle

\setcounter{footnote}{3}

\section{Introduction}\label{S:Intro}
In a seminal 1978 paper \cite{gl}, Gordon--Litherland describe a symmetric pairing $\G_F$ on the first homology group of a spanning surface $F$ for a knot $L\subset S^3$, proving in particular that the signature $\sigma(\G_F)$ minus half of the boundary slope of $F$ is an invariant of $L$, namely the knot signature: $\sigma(L)=\sigma(\G_F)-\frac{1}{2}s(F)$.  (Everything works for links too, but with more nuance.) 
Their paper is remarkable in several ways.  One is that their surfaces $F$ are {\it unoriented} and often nonorientable, a novelty at the time (Kauffman--Taylor had recently proven the main results in the oriented case \cite{kt}).
Another
 is the dual structure of their exposition, which reconciles accessibility with sophistication: after using double-branched covers of the 4-ball to prove their main results and convey their deep significance, they take an alternative, ``down-to-earth" approach. The main idea {of that approach} is that 
all spanning surfaces for a given link $L\subset S^3$ are related by attaching and removing tubes and trivial M\"obius bands in the manner shown in Figure \ref{Fi:Moves} (see \textsection\ref{S:equiv} for more precise descriptions of these moves; we also call the latter moves (geometric) {\it kinking} and {\it unkinking moves}), and that these moves change Gordon-Litherland pairings, and the {\it Gordon-Litherland matrices} that represent them, in predictable ways, namely:
\begin{figure}[t]
\begin{center}
\scalebox{.95}{\raisebox{.04\textwidth}{\includegraphics[height=.12\textwidth]{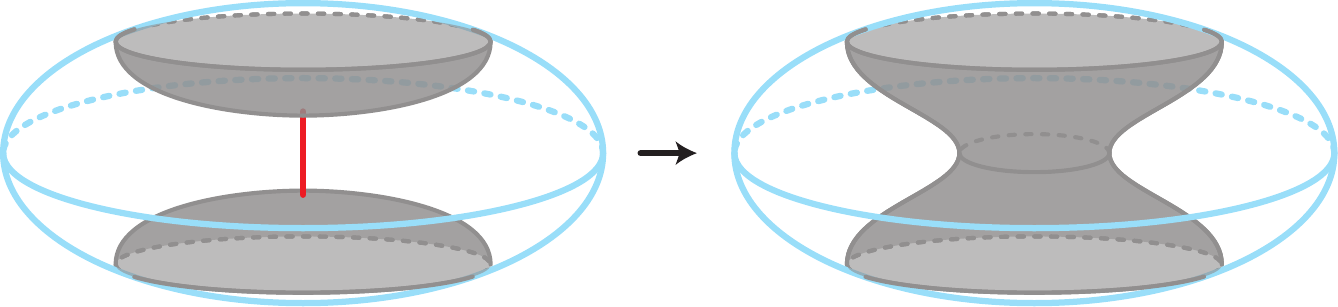}} \hspace{.5in} \includegraphics[height=.2\textwidth]{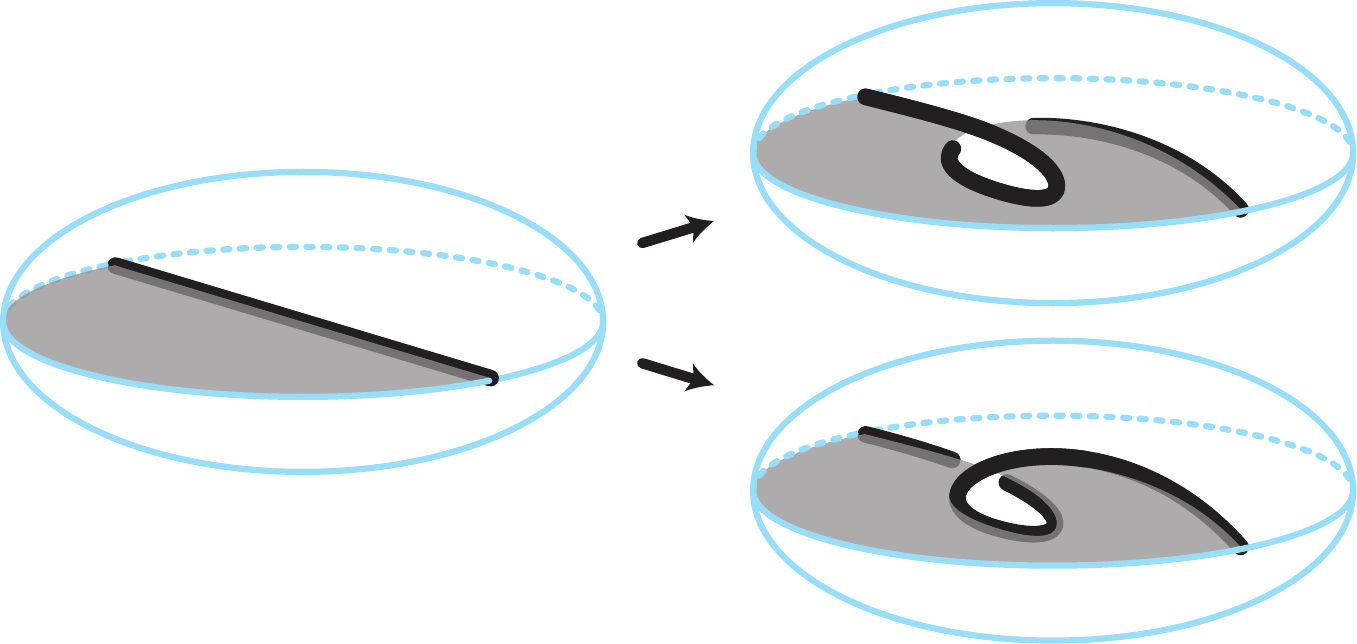}}
\caption{Attaching a tube (left) and positive (top-right) and negative (bottom-right) trivial M\"obius bands to a spanning surface.}
\label{Fi:Moves}
\end{center}
\end{figure}
\begin{equation}\label{E:G'G''}
G{\leftrightarrow}\bbm G&\vb&\0\\ 
\vb^T&0&1\\
\0^T&1&0 \ebm
\hspace{.4in}\text{and}\hspace{.4in}
G{\leftrightarrow}\bbm G&\0\\ 
\0^T&\pm1 \ebm{=G\oplus[\pm 1]}
\end{equation}
for some $\vb\in\Z^n$.  
Similar arguments apply in the oriented setting, as
all Seifert surfaces for a given oriented link $L\subset S^3$ are related by attaching and removing just tubes \cite{bnfk,levine}. 

On the other hand, all checkerboard surfaces for all connected diagrams of a given link $L\subset S^3$ are related by just (geometric) kinking and unkinking moves, i.e. by attaching and removing trivial M\"obius bands, no tubes needed \cite{encyc}---see Theorem \ref{T:encyc}. This implies that all Gordon-Litherland matrices for these surfaces are {\bf kink-equivalent}, %
\footnote{In the literature on linking forms, this relation is called {\it stable equivalence.} See Remarks \ref{R:Linking_Forms}--\ref{R:Stabilization}.} denoted $\sim$, meaning that they are related by unimodular congruence and what we call (matrix) ``kinking,'' and ``unkinking'' moves: %
\footnote{Note that nullity and the absolute value of determinant are invariant under kink-equivalence.}

\begin{equation}\label{E:Kink}
G\leftrightarrow P^TGP\text{ (}P\text{ unimodular)}\hspace{.4in}\text{and}\hspace{.4in}
G\leftrightarrow\bbm G&\0\\ 
\0^T&\pm1 \ebm{=G\oplus[\pm 1]}.
\end{equation}

It is important to note that when $L$ has a spanning surface $F$ with a Gordon-Litherland matrix of the form $G={G'\oplus[\pm 1]}$ it need not follow that $F$ actually (geometrically) admits an {unkinking move} of the appropriate sign.  See Example \ref{Ex:Naive}.  If $F$ admits no such move, we {say that the (matrix) unkinking move $G'\oplus[\pm1]\to G'$ describes a {\it fake unkinking move} on $F$}. 

Nevertheless, the notion of kink-equivalence of matrices is intriguing.
For example, having a spanning surface that is kink-equivalent to positive- and negative-definite spanning surfaces is an extremely restrictive condition on links in $S^3$. In fact, it is precisely the alternating condition; see Theorem \ref{T:Greene} \cite{greene} and Remark \ref{R:Greene}.  How much less restrictive is the {(naively) }analogous condition on Gordon-Litherland matrices?

\begin{question}\label{Q:1}
Is there a non-alternating link $L\subset S^3$ with a spanning surface $F$ with a Gordon-Litherland matrix that is kink-equivalent to positive- and negative-definite matrices? If so, which links have such $F$?
\end{question}

More fundamentally:

\begin{question}\label{Q:2}
Which symmetric integer matrices $A$ are kink-equivalent to both positive- and negative-definite matrices? To  positive- and negative-semidefinite matrices?
\end{question}

\begin{rem}\label{R:Uni}
In the special case where $A$ is unimodular, Theorems 4 and 5 in \textsection2.2 of \cite{serre} (Theorem 1.2.14 of \cite{gs}) give a positive answer to Question \ref{Q:2}. Indeed, these structure theorems imply that $A\oplus[1]$ is (unimodularly) congruent to $I_m\oplus -I_n$ for some $m,n$ and thus is kink-equivalent to $[\pm1]$.
\end{rem}

What can one say in general, though? We answer Questions \ref{Q:1} and \ref{Q:2} completely:

\begin{theorem}\label{T:MainMatrix}
Every kink-equivalence class of symmetric integer matrices contains positive- and negative-semidefinite representatives. 
Thus, for a symmetric matrix $G\in\Z^{n\times n}$, the following are equivalent:
\begin{itemize}
\item $G$ is kink-equivalent to a positive-definite matrix,
\item $G$ is kink-equivalent to a negative-definite matrix,
\item $G$ is nonsingular. %
\footnote{Throughout, ``nonsingular'' matrices have nonzero determinant but are not necessarily invertible over $\mathbb{Z}$.}
\end{itemize}
Moreover, given a nonsingular symmetric matrix $G\in\Z^{n\times n}$,
writing $n_\pm=\frac{1}{2}(n\pm\sigma(G))$, %
\footnote{These count the number of positive 
and negative eigenvalues of $G$; $\sigma(G)$ is the signature of $G$.}
there are $\pm$-definite integer 
matrices $A_\pm$ that satisfy the unimodular congruences  $G\oplus I_{4n_-}\cong A_+\oplus -I_{n_-}$ and $G\oplus-I_{4n_+}\cong A_-\oplus I_{n_+}$.%
\footnote{That is, $G$ can be changed to, say, a positive-definite matrix via (at most) $4n_-=2(n-\sigma(G))$ positive kinking moves, a congruence, and then $n_-$ negative unkinking moves.}
\end{theorem}

As we will state more precisely in Theorem \ref{T:MainKnotLink}, it follows that each knot in $S^3$ is ``alternating up to fake unkinking" in some sense. 

Theorem \ref{T:MainMatrix} also has implications for any (orientable) 4-manifold $M$ via its intersection pairing $Q$. If $M$ is closed, then $Q$ is unimodular by Poincar\'e duality. If $M$ is also simply connected, then Serre's classification described in Remark \ref{R:Uni}, together with a theorem of Freedman \cite{freedman} (Theorem 1.2.27 of \cite{gs}), implies the following, e.g.:

\begin{theorem}
Given any simply connected, 
closed 4-manifold $M$, 
$M\oplus\CP^2$ and $M\oplus\overline{\CP^2}$ are {homeomorphic} to connect sums of $\CP^2$'s and $\overline{\CP^2}$'s.  
\end{theorem}

Our Theorem \ref{T:MainMatrix} provides a new perspective on this well-known fact.  In more general circumstances, Theorem \ref{T:MainMatrix} implies what seems to be a new result:

\begin{theorem}\label{T:Main4Manifold1}
Given a 
4-manifold $M$, where $H_2(M)$ has rank $n$ and nonsingular intersection pairing, $\cdot$, write $n_\pm=\frac{1}{2}(n\pm\sigma(M))$. There are simply connected, 
$\pm$-definite 4-manifolds $M_\pm$ (not necessarily closed) and these isomorphisms of intersection pairings on blowups and anti-blowups: %
\footnote{These are isomorphisms of second homology groups that respect intersection numbers.}
\[(H_2(M\underset{i=1}{\overset{4n_+}{\#}}
\overline{\CP^2}),\cdot)\cong(H_2(M_-\underset{i=1}{\overset{
n_+}{\#}}\CP^2),\cdot)\]
and
\[(H_2(M\underset{i=1}{\overset{4n_-}{\#}}{\CP^2}),\cdot)\cong(H_2(M_+\underset{i=1}{\overset{n_-}{\#}}\overline{\CP^2}),\cdot).%
\]
This is also true if one removes ``nonsingular" and replaces ``definite" with ``semidefinite."
\end{theorem}
 
Theorem \ref{T:MainMatrix} also has a direct implication for quadratic forms, but only those whose cross-terms all have even coefficients.  See Theorem \ref{T:MainQuadratic}.  This motivates an inquiry quite distinct from our original setting: can we extend Theorem \ref{T:MainMatrix} to matrices with half-integers off the diagonal and thus remove the cross-term condition from Theorem \ref{T:MainQuadratic}? Indeed, we prove more generally that Theorem \ref{T:MainMatrix} extends to the case of rational matrices, at the cost of increasing the bound on the number of stabilizations required from $4n$ to $5n$. See Theorem \ref{T:RationalExtension}. 
Writing the quadratic form $q_0:x\mapsto x^2$ thus leads to the following theorem.

\begin{theorem}\label{T:MainQuadratic2}
Let $q:\Z^n\to \Q$ be a nonsingular quadratic form, and write $n_\pm=\frac{1}{2}(n\pm\sigma(q))$. There are $\pm$-definite quadratic forms $q_\pm$ that satisfy the unimodular congruences $q\oplus (q_0)^{  5n_-}\cong q_+\oplus (-q_0)^{n_-}$ and $q\oplus (-q_0)^{  5n_+}\cong q_-\oplus (q_0)^{n_+}$.
\end{theorem}

{\noindent{\bf Outline: } \textsection\ref{S:Back} gives background, \textsection\ref{S:Main} explores Question \ref{Q:2} and proves Theorem \ref{T:MainMatrix}, but not before entertaining a brief diversion (see Question \ref{Q:CCT} below), \textsection\ref{S:Final} discusses implications and further questions, and \textsection\ref{S:Quadratic} pursues the inquiry described in the previous paragraph.  

\begin{question}\label{Q:CCT}
Given a positive-definite matrix $G_+\in\Z^{n\times n}$, must there be a matrix $C\in \Z^{n\times m}$ (for some $m\geq n$) such that $CC^T=G_+$?
\end{question}

In \textsection\ref{S:CCT}, we explain the pertinence of Question \ref{Q:CCT} to Question \ref{Q:2}.  Although the answer Question \ref{Q:CCT} turns out to be no (Theorem \ref{thm: counterexample matrix} gives a counterexample), the diversion yields:

\begin{C:square}
Every matrix of the form $G=I+CC^T$ for a symmetric integer matrix $C$ is kink-equivalent to $-G$. In particular, this holds for every matrix  
$\bbme{n^2+1}$, $n\in\Z$.
\end{C:square}

\section{Background}\label{S:Back}

We work in the piecewise-linear category. Except where stated otherwise, we assume that links are in $S^3$ and that their diagrams are on $S^2$ and connected.  By ``congruence," we always mean ``unimodular congruence."
 
\subsection{Linear algebra and number theory}

Here we introduce some well{-}known results from linear algebra and number theory that are used in the paper.

\begin{theorem}[Sylvester's inertia theorem]\label{T:Syl}
For a symmetric real matrix $A$ and invertible matrix $P$, ${P^TAP}$ has the same number of positive, negative, and zero eigenvalues as $A$. 
\end{theorem}

A vector $\mathbf{x}\in\Z^n$ is called {\bf primitive} if the greatest common divisor of its entries is one.  {An integer matrix $A$ is {\bf unimodular} if it is invertible over $\Z$; two equivalent conditions are that the columns of $A$ form a $\Z$-basis for $\Z^n$, or that $\det(A)=\pm1$.}

\begin{fact}[Consequence of Rado's Lemma \cite{rado}]
\label{F:Primitive}
Any primitive vector $\mathbf{x}\in\Z^n$ extends to a basis for $\Z^n$, i.e. there is a unimodular matrix whose first column is $\mathbf{x}$.
\end{fact}

{The bounds in Theorem~\ref{T:MainMatrix} will come from the following classical result:}

\begin{theorem}[Lagrange's Four Squares Theorem]\label{thm: foursq}
Let $x\geq0$ be an integer. Then there exist integers $a,b,c,d$ such that $x=a^2+b^2+c^2+d^2$. 
\end{theorem}

\subsection{Checkerboard surfaces for knots and links in $S^3$}

Given a diagram $D$ of a link $L$, one may shade the complementary regions of $D$ in $S^2$ black and white in checkerboard fashion, so that regions of the same shade meet only at crossing points.
See Figure \ref{Fi:Checkerboards}, left and center.
One may then construct spanning surfaces $B$ and $W$ for $L$ such that $B$ projects to the black regions, $W$ projects to the white, and $B$ and $W$ intersect in {vertical arcs} which project to the crossings of $D$.  Call the surfaces $B$ and $W$ the {\bf checkerboard surfaces} from $D$. 
These are {\bf spanning surfaces} for $L$: compact connected surfaces in $S^3$ with boundary equal to $L$.  (Not every spanning surface can be realized as a checkerboard surface, though
---see \cite{kob}, e.g.)

More generally, one constructs a (Kauffman) {state} $x$ of $D$ by smoothing each crossing in one of two ways, $\raisebox{-.02in}{\includegraphics[width=.125in]{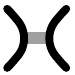}}
\overset{\color{Gray}{_{{A}}}\color{black}}{\longleftarrow}\raisebox{-.02in}{\includegraphics[width=.125in]{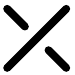}}
\overset{_{{B}}}{\longrightarrow}\raisebox{-.02in}{\includegraphics[width=.125in]{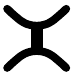}}$, and for any state $x$ of $D$ one can construct a spanning surface $F_x$ for $L$, called a {\bf state surface}, by attaching a disk to each state circle (typically with all disks' interiors {on the same side} of $S^2$) and attaching a half-twisted band at each crossing.  Note that both checkerboard surfaces of $D$ are particular examples of state surfaces.  In fact, every state surface of $D$ is isotopic to a checkerboard surface of some diagram \cite{thesis,essence}.  

The rank $\beta_1(F)$ of the first homology group of a spanning surface $F$ counts the number of ``holes'' in $F$. When $F$ is connected, $\beta_1(F)=1-\chi(F)$ counts the number of cuts along disjoint, properly embedded arcs required to reduce $F$ to a disk.

\subsection{Definite surfaces and alternating links}\label{S:Goeritz}

\begin{figure}
\begin{center}
\includegraphics[width=.2\textwidth]{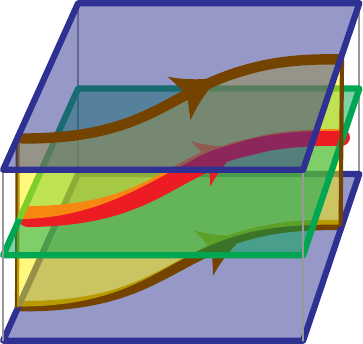}\caption{A curve $\red{\gamma}$ on $\FG{F}$, with $\brown{\wt{\gamma}=p^{-1}(\gamma)}$ on $\Navy{\wt{F}}$}
\label{Fi:GL}
\end{center}
\end{figure}

Given a surface $F$ spanning a link $L$, take a closed regular neighborhood $\nu F$ in the link exterior $S^3\cut\nu L$
with projection ${p}:{\nu}F\to F$
such that $p^{-1}(\partial F)=\nu F\cap\partial\nu L$; see Figure \ref{Fi:GL}.  
Denote the {frontier} $\wt{F}=\partial\nu F{\cut}\partial\nu L$ %
\footnote{Thus, the restriction $p|_{\wt{F}}:\wt{F}\to F$ is a 2:1 covering map, $\wt{F}$ is orientable, and $\wt{F}$ is connected if and only if $F$ is connected and nonorientable.} and
transfer map $\tau:H_1(F)\to H_1(\widetilde{F})$. %
\footnote{Given any 
$g\in H_1(F)$, choose an oriented multicurve $\gamma\subset\text{int}(F)$ representing $g$, 
denote $\widetilde{\gamma}=\partial({p}^{-1}(\gamma))$, and %
orient $\widetilde{\gamma}$ following $\gamma$; then, $\tau(g)=[\widetilde{\gamma}]$.} The {\bf Gordon-Litherland pairing} \cite{gl}
\[\langle\cdot,\cdot\rangle:H_1(F)\times H_1(F)\to\Z\]
 is the symmetric, bilinear mapping 
given by the linking number
\[\langle a,b\rangle=\text{lk}(a,\tau(b)).\]
Any projective homology class $g=[\gamma]\in H_1(F)/\pm$ has a well-defined {self-pairing} $\lb g \rb=\langle g,g\rangle$; the {framing} of $\gamma$ in $F$ is given by $\frac{1}{2}\lb g \rb$. When $L$ is a knot, the framing of (a pushoff of) $L$ in $F$ is called the {\bf slope} of $F$, written $s(F)$.  When $L$ has multiple components, the slope $s(F)$ is the sum of the components' framings in $F$.

Given a 
basis $\mathcal{B}=(a_1,\hdots,a_m)$ for $H_1(F)$, the {\bf Gordon-Litherland matrix} $G=(x_{ij})\in\Z^{m\times m}$ given by $x_{ij}=\langle a_i,a_j\rangle$ represents $\langle\cdot,\cdot\rangle$ with respect to $\mathcal{B}$. %
\footnote{That is, any $\displaystyle{y=\sum_{i=1}^my_ia_i}$ and $\displaystyle{z=\sum_{i=1}^mz_ia_i}$ satisfy
$\displaystyle{\langle y,z\rangle=[y_1~\cdots~y_m]G[z_1~\cdots~
z_m
]^T.}$}
When $F$ is the shaded checkerboard surface from a diagram $D$, the boundaries of (all but one of) the unshaded regions $S^2\cut D$, each oriented counterclockwise, provide a convenient basis $\mathcal{B}$ for $H_1(F)$.  The Gordon-Litherland matrix for this surface with this basis is called a {\bf Goeritz matrix}.  Figure \ref{Fi:Checkerboards}, right, shows an illustrative example.
The signature of $G$ is called the {signature of $F$} and is denoted $\sigma(F)$.  Gordon-Litherland show that the quantity $\sigma(F)-\frac{1}{2}s(F)$ is independent of $F$, and in fact equals the Murasugi invariant $\xi(L)$, which is the average signature of $L$ across all orientations. 

They also show that $\sigma(F)$ is the signature of the 4-manifold obtained by pushing the interior of $F$ into the interior of the 4-ball $B^4$, while fixing $\partial F$ in $\partial B^4=S^3$, and taking the double-branched cover of $B^4$ along this surface. In particular, when $L$ is a knot, $\xi(L)$ is the signature of $L$ and of the 
double-branched cover of $B^4$ along any perturbed Seifert surface.

A spanning surface $F$ is
{\bf positive-definite} if $\lla g,g\rra>0$ for all nonzero $g\in H_1(F)$ \cite{greene}.
Equivalently, $F$ is positive-definite if and only if
$\sigma(F)=\beta_1(F)$.
{\bf Negative-definite} surfaces are defined analogously. Greene characterizes alternatingness in terms of the existence of positive- and negative-definite pairs of spanning surfaces:

\begin{figure}[t]
\begin{center}
\raisebox{-0.5\height}{\includegraphics[height=1in]{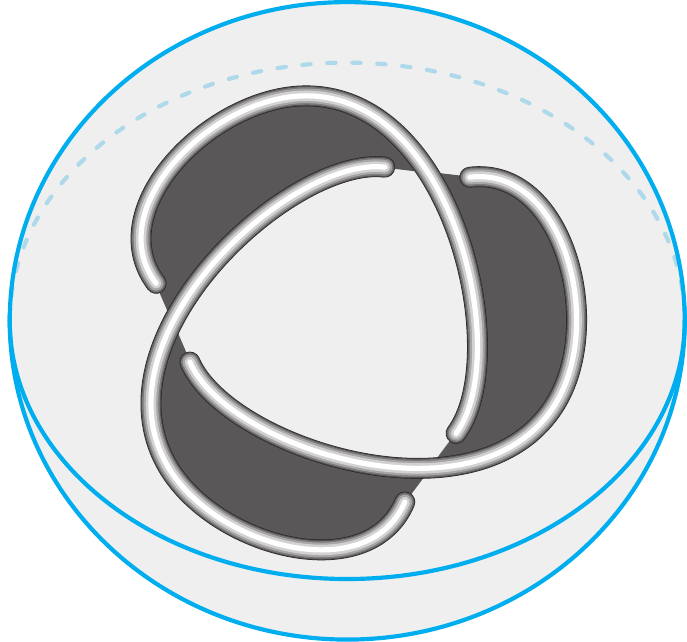}}$~\leadsto~\bbme{3}$\hspace{1in}\raisebox{-0.5\height}{\includegraphics[height=1in]{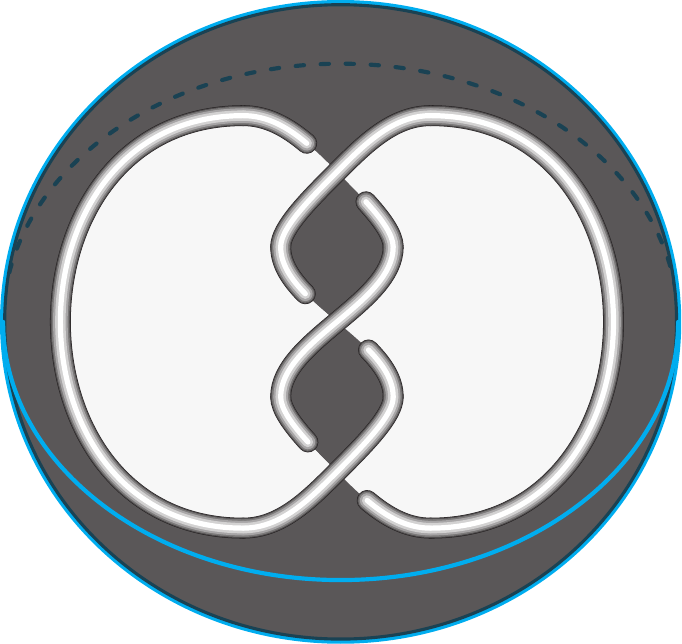}}$~\leadsto~\bbme{-2&1\\1&-2}$
\quad
\labellist \small
 \pinlabel {${R_1}$} [c] at 220 410
 \pinlabel {${R_2}$} [c] at 360 230
 \pinlabel {${R_3}$} [c] at 160 160
 \tiny \pinlabel {$\red{{a_1}}$} [c] at 80 400 \pinlabel {$\Yel{{a_2}}$} [c] at 465 145 \pinlabel {$\Navy{{a_3}}$} [c] at 70 70 \pinlabel {$\boldsymbol{+}$} [c] at 30 190 \pinlabel {$\boldsymbol{+}$} [c] at 220 50 \pinlabel {$\boldsymbol{+}$} [c] at 240 200 \pinlabel {$\boldsymbol{+}$} [c] at 470 240 \pinlabel {$\boldsymbol{+}$} [c] at 302 292 \pinlabel {$\boldsymbol{+}$} [c] at 200 475 \pinlabel {$\boldsymbol{+}$} [c] at 380 110 \endlabellist\includegraphics[height=1.5in]{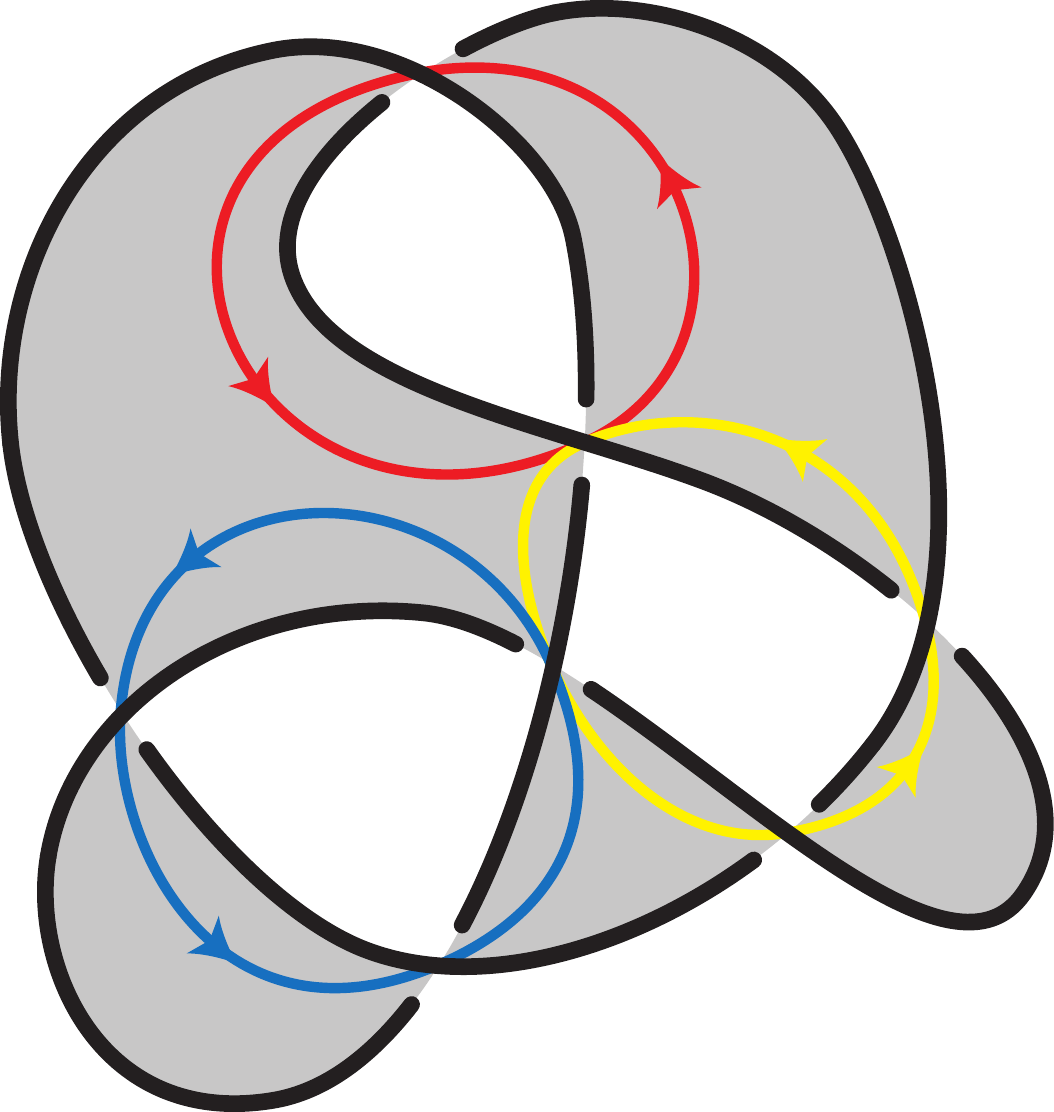}\raisebox{.5in}{$\leadsto\begin{bmatrix}\red{2}&\Orange{-1}&\violet{0}\\\Orange{-1}&\Yel{4}&\FG{-1}\\\violet{0}&\FG{-1}&\Navy{3}\end{bmatrix}$}
\caption{Three checkerboard surfaces and their Gordon-Litherland matrices.}
\label{Fi:Checkerboards}
\end{center}
\end{figure}

\begin{theorem}[\cite{greene}]\label{T:Greene}
A link $L\subset S^3$ is alternating if and only if it has spanning surfaces $F_+$ and $F_-$ whose respective Gordon-Litherland matrices are positive- and negative-definite.  Moreover, given such $F_\pm$, $L$ has a diagram whose checkerboard surfaces are isotopic to $F_\pm$.
\end{theorem}

\begin{rem}\label{R:Greene}
As noted in the introduction, it follows that $L$ is alternating if and only if it has a spanning surface $F$ that is kink-equivalent to positive- and negative-definite surfaces $F_+$ and $F_-$.  Indeed, if such surfaces exist, then $L$ is alternating by Theorem \ref{T:Greene}. Conversely, if $L$ is alternating, then Theorem \ref{T:Greene} affords positive- and negative-definite spanning surfaces $F_+$ and $F_-$ for $L$, and Theorem \ref{T:encyc} implies that $F_+$ and $F_-$ are kink-equivalent.
\end{rem}

As described in the introduction, Question \ref{Q:1} asks how much of Greene's characterization is happening purely on the level of linear algebra. 

\subsection{Equivalences}\label{S:equiv}

Traditionally, two symmetric integer matrices are called {\bf $\boldsymbol{S}$-equivalent} if they are related by congruence moves, given by $G\leftrightarrow PGP^T$, and (matrix) {\it tubing} and {\it untubing moves}, given by

\[G\leftrightarrow\bbme{ G&\vb&\0\\ 
\vb^T&0&1\\
\0^T&1&0 }.\]

{\bf $\boldsymbol{S^*}$-equivalence} of matrices is defined the same way, except that (matrix) kinking and unkinking moves are also allowed, $G\leftrightarrow\bbmes{ G&\0\\
\0^T&\pm 1}=G\oplus[\pm1].$
These notions of $S$- and $S^*$-equivalence find motivation in the following theorems, in which, given a spanning surface $F$,
\begin{itemize}
\item ``attaching a tube'' to $F$ consists of surgering $F$ along a properly embedded arc in the exterior of $F$, as shown left in Figure \ref{Fi:Moves};
\item ``removing a tube'' from $F$ is the reverse operation, also called a compression move, which consists of surgering $F$ along a properly embedded disk in the exterior of $F$;
\item ``attaching a trivial M\"obius band" to $F$ consists of taking a boundary sum of $F$ with either of the M\"obius bands that span the unknot (this is the same as a (geometric) kinking move);
\item ``removing a trivial M\"obius band" from $F$ is the reverse operation, also called a (geometric) unkinking move, which consists of decomposing $F$ as a boundary sum in which one summand is a M\"obius band that spans the unknot and then deleting this summand from $F$.
\end{itemize}
Also, we call two spanning surfaces {\bf kink-equivalent} if they are related by a sequence of kinking and unkinking moves.

\begin{theorem}[Theorem 1 of \cite{bnfk}; \cite{levine}]\label{T:tube}
All Seifert surfaces for a given oriented link $L\subset S^3$ are related by attaching and removing tubes.
\end{theorem}

It is surprisingly hard to find a clear statement of Theorem \ref{T:tube} in the literature prior to \cite{bnfk}, although the abstract of that paper asserts that this result had already long been well known, citing \cite{levine} as evidence of this.

\begin{theorem}[Theorem 11 of \cite{gl}]\label{T:tubekink}
All spanning surfaces for a given link $L\subset S^3$ are related by attaching and removing tubes and trivial M\"obius bands.
\end{theorem}

Bar-Natan--Fulman--Kauffman describe the following elementary proof of Theorem \ref{T:tube} in \cite{bnfk}, and Yasuhara adapts that proof to give a similar, simple proof of Theorem \ref{T:tubekink} in \cite{yas}.
They argue as follows.  Given two Seifert surfaces $F_1,F_2$ for $L$, isotope each $F_i$ in a separate copy of $S^3$ into {\it disk-band form}. 
Add tubes where bands cross to change each $F_i$ into an oriented checkerboard surface $F'_i$ for some connected diagram $D_i$ of $L$; 
thus, each $F'_i$ is the ``algorithmic Seifert surface" from $D_i$.  Finally, change $D_1\to D_2$ using Reidemeister moves (without disconnecting the diagram); each move changes the algorithmic Seifert surface by isotopy and/or tubing.
In \cite{encyc}, Kindred 
adapts these arguments to prove the following related fact:

\begin{theorem}[\cite{encyc}]\label{T:encyc}
All state surfaces, and in particular all checkerboard surfaces, for all connected diagrams of a given link $L\subset S^3$ are kink-equivalent. 
\end{theorem}

We call two spanning surfaces {\bf kink-equivalent} if they are related in this way. 
Kindred's proof essentially proceeds as follows.

\begin{figure}
\begin{center}
\labellist
  \pinlabel {$\cup$} at 72 110
 \pinlabel {$=$} at 222 110
 \pinlabel {$\longrightarrow$} at 422 110
 \pinlabel {$\longrightarrow$} at 622 110
 \pinlabel {$\longrightarrow$} at 822 110
 \pinlabel {$=$} at 1022 110
 \pinlabel {$\cup$} at 1180 110
\tiny
 \pinlabel {add} at 422 133
 \pinlabel {kink} at 422 93
 \pinlabel {slide blue} at 622 133 \pinlabel {over yellow} at 622 93 \pinlabel {isotope} at 822 133 \endlabellist \includegraphics[width=\textwidth]{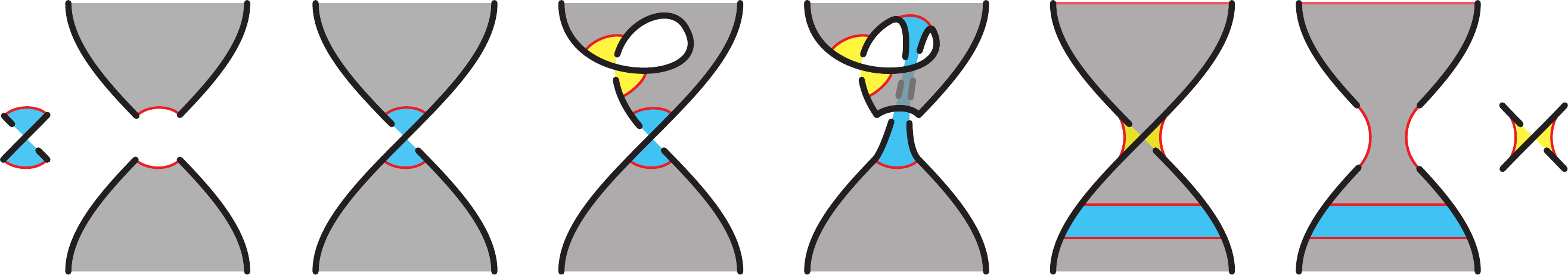}
\caption{Reversing the smoothing at a crossing in a state surface has the same effect as performing a kinking or unkinking move.}
\label{Fi:kink1}
\end{center}
\end{figure}

\begin{proof}
First, we claim that all state surfaces from {\it a given diagram} $D$, and in particular both checkerboard surfaces from $D$, are kink-equivalent. It suffices to prove this for surfaces from states whose smoothings differ at a single crossing.  Figure \ref{Fi:kink1} does just this.

\begin{figure}
    \centering
    \raisebox{.2in}{\includegraphics[scale=.35]{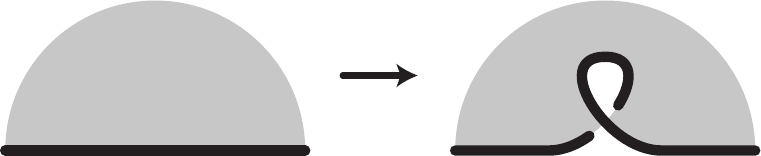}}\hfill
    \includegraphics[scale=.35]{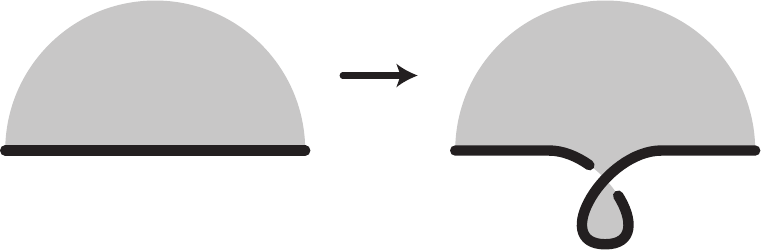}\hfill
    \includegraphics[scale=.35]{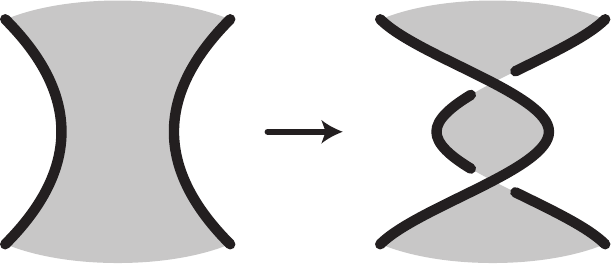}
\caption{Kink-equivalent checkerboard surfaces under R1 and R2 moves}
\label{Fi:R1R2}
\end{figure}

\begin{figure}
\begin{center}
\labellist
 \pinlabel {$\longrightarrow$} at 1252 140
 \pinlabel {$\longrightarrow$} at 382 120
 \pinlabel {$\longrightarrow$} at 815 130
\tiny
 \pinlabel {add} at 382 150
 \pinlabel {kink} at 382 100
 \pinlabel {slide blue} at 815 160 \pinlabel {over yellow} at 815 107 \pinlabel {isotope} at 1252 170 \endlabellist \includegraphics[width=\textwidth]{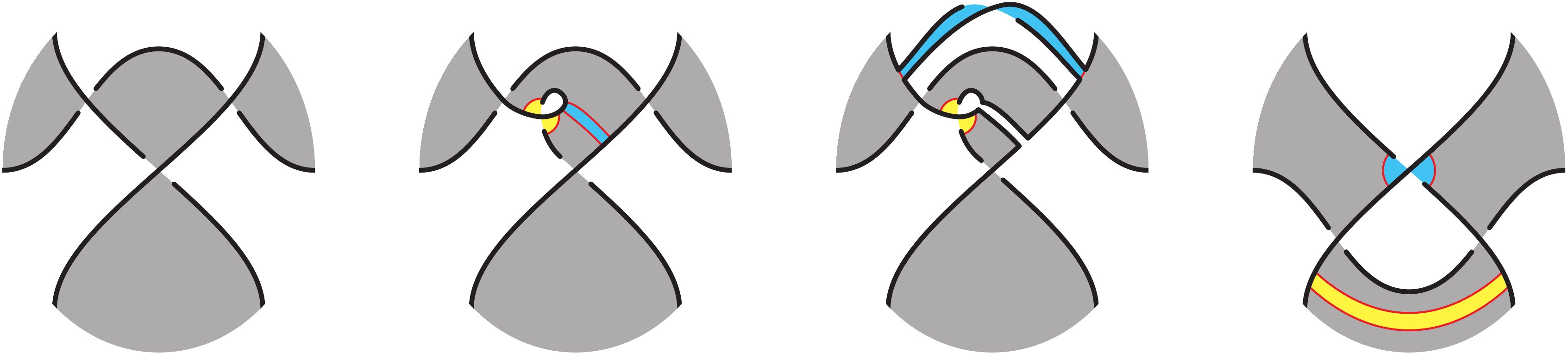}
\caption{A Reidemeister 3 move alters a checkerboard surface by a kinking or unkinking move.}
\label{Fi:kink2}
\end{center}
\end{figure}

The rest of the proof comes down to three observations. First, Theorem 1 of \cite{coward} implies that any two connected diagrams of the same link are related by a sequence of Reidemeister moves between connected diagrams.  Second, if $D\to D'$ is a Reidemeister 1 or 2 move, then (at least) one of the checkerboard surfaces from $D$ is related to a checkerboard surface from $D'$ (see Figure \ref{Fi:R1R2}) by isotopy and possibly a single kinking move; it then follows from the previous paragraph that all state surfaces from $D$ are kink-equivalent to all those from $D'$. Third, as shown in Figure \ref{Fi:kink2}, if $D\to D'$ is a Reidemeister 3 move, then adding the correct kink to the correct checkerboard surface from $D$ gives a surface isotopic to a checkerboard surface from $D'$.
\end{proof}

Theorem \ref{T:encyc} implies that, given checkerboard surfaces $F_1$ and $F_2$ of diagrams $D_1$ and $D_2$ of a link $L$, their Gordon-Litherland matrices $G_1$ and $G_2$ are related by a sequence of congruence, ``kinking,'' and ``unkinking'' moves:

\begin{equation}\label{E:NaiveKink}
G\leftrightarrow P^TGP\text{ (}P\text{ unimodular)}\hspace{.4in}\text{and}\hspace{.4in}
G\leftrightarrow G\oplus[\pm1].
\end{equation}

We say that $G_1$ and $G_2$ are {\bf  kink-equivalent}, and we write $G_1\sim G_2$. 

\subsection{Fake unkinking moves}\label{S:Fake}

We mentioned earlier that not every unkinking move $G\oplus[\pm1]\to G$ on a Gordon-Litherland matrix can be realized geometrically as an unkinking move on a given spanning surface.  We will give an explicit example of this phenomenon shortly.  To justify our claims in that example, we will need an obstruction, which will use:

\begin{theorem}[Theorem 5.4 of \cite{ak}]
Let $D$ be an alternating diagram of a link $L\subset S^3$, and let $F$ be a spanning surface for $L$.  Then $F$ has the same homeomorphism type and slope as some state surface from $D$, perhaps with tubes or trivial M\"obius bands attached.  
\end{theorem}

In particular, this solves the geography problem for any alternating link $L$, which asks which pairs $(s(F),\beta_1(F))$ are realized by its spanning surfaces.  The answer is obtained most efficiently by plotting a point $(s(F),\beta_1(F))$ for each {\it adequate} state surface from $D$ (a state is adequate if each crossing joins distinct state circles; Figure \ref{Fi:kink1} implies that every state surface is isotopic to an adequate state surface with trivial M\"obius bands attached), extending upward from each plotted point by two rays with slopes $\pm\frac{1}{2}$, and shading above all rays. Figure \ref{Fi:Geog} shows the geography of (spanning surfaces for) the $8_5$ knot.

\begin{example}\label{Ex:Naive}
A checkerboard surface $F_{12}$ (shown 
in Figure \ref{Fi:Geog}) for a minimal diagram of the $8_5$ knot has $\beta_1(F_{12})=2$ and $s(F_{12})=12$, and adding a negative kink to $F_{12}$ yields a surface $F_{10}$ with $\beta_1(F_{10})=3$ and $s(F_{10})=10$.  Moreover, $F_{12}$ has Gordon-Litherland matrix $\bbmes{5&3\\3&6}$, so Gordon-Litherland matrices for $F_{10}$ include $\bbmes{5&3\\3&6}\oplus[-1]$ and the congruent matrix\\

\begin{figure}
\labellist
\small\hair 4pt
 \pinlabel {$\boldsymbol{\beta_1}$} [c] at 185 320
 \pinlabel {$\boldsymbol{2}$} [c] at 184 80
 \pinlabel {$\boldsymbol{4}$} [c] at 184 152
 \pinlabel {$\boldsymbol{6}$} [c] at 184 224
 \pinlabel {$\boldsymbol{8}$} [c] at 184 296
 \pinlabel {$\textbf{s}$} [l] at 550 -10
 \pinlabel {$\boldsymbol{4}$} [c] at 242 -10
 \pinlabel {$\boldsymbol{8}$} [c] at 315 -10
 \pinlabel {$\boldsymbol{12}$} [c] at 385 -10
 \pinlabel {$\boldsymbol{-4}$} [c] at 95 -10
 \pinlabel {$\boldsymbol{\sigma}$} [l] at 550 20
 \pinlabel {$\boldsymbol{-2}$} [c] at 242 20
 \pinlabel {$\boldsymbol{0}$} [c] at 315 20
 \pinlabel {$\boldsymbol{2}$} [c] at 385 20
 \pinlabel {$\boldsymbol{-6}$} [c] at 95 20
\endlabellist
\includegraphics[width=.8\textwidth]{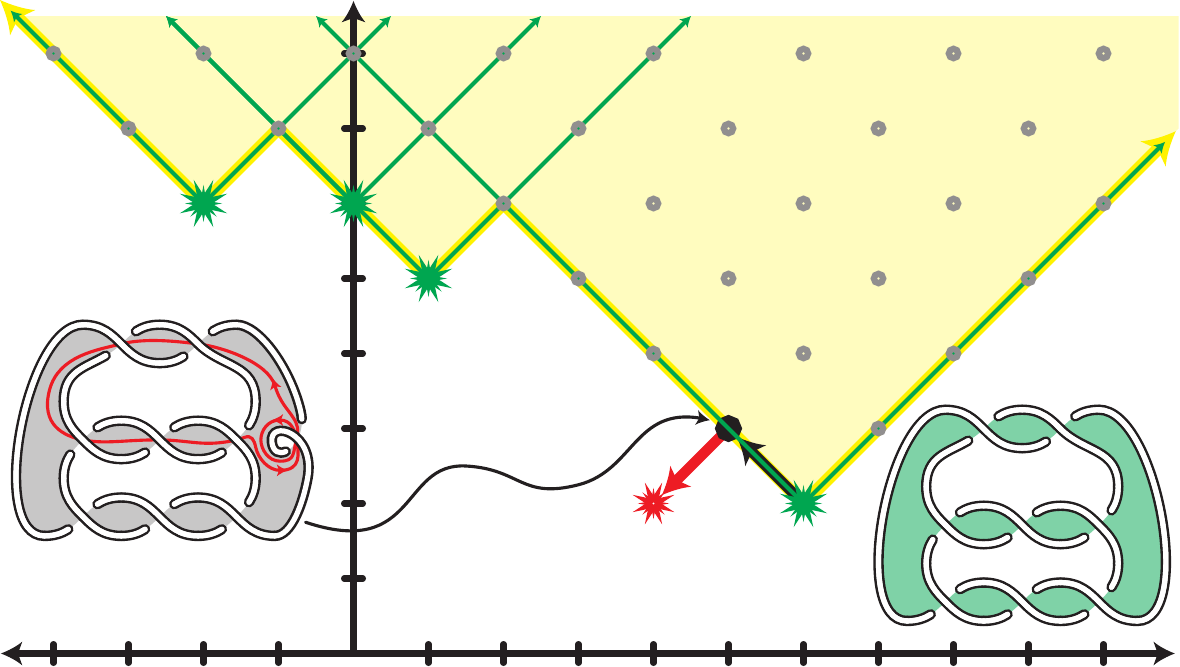}\caption{Geography of the $8_5$ knot (adequate surfaces plotted as green stars) and its implications in Example \ref{Ex:Naive}}
\label{Fi:Geog}
\end{figure}

\noindent\resizebox{1\linewidth}{!}{$\bbme{1&0&0\\-3&1&0\\2&0&1}\bbme{1&0&2\\0&1&0\\0&0&1}\bbme{5&3&0\\3&6&0\\0&0&-1}\bbme{1&0&2\\0&1&0\\0&0&1}^T\bbme{1&0&0\\-3&1&0\\2&0&1}^T=\bbme{1&0&0\\0&-3&6\\0&6&-5},$}\\

\noindent
which admits a positive unkinking move.  Yet, $F_{10}$ admits no geometric positive unkinking move (e.g. along the red 1-framed curve in the figure), because such a move would yield a spanning surface $F_8$ for the $8_5$ knot with $\beta_1(F_{8})=2$ and $s(F_{8})=8$, and Figure \ref{Fi:Geog} shows that no such surface exists, for geography reasons.  Therefore, the positive unkinking move $[1]\oplus\bbmes{-3&6\\6&-5}\to\bbmes{-3&6\\6&-5}$ on the Gordon-Litherland matrix for $F_{10}$ does, in fact, describe a {\it fake} unkinking move on $F_{10}$. Also see Remark \ref{rem:fake}.
\end{example}

\section{Main results}\label{S:Main}

\subsection{The transpose product approach}\label{S:CCT}

One of the defining properties of positive-definite matrices over $\R$ is that they take the form $M=CC^T$ for some matrix $C$ (and in fact one may take $C$ itself to be positive-definite by orthogonally diagonalizing $M=UDU^T$ and setting $C=U\sqrt{D}U^T$, in which case $C$ is referred to as the square root of $M$). It is clear by examining the one-by-one matrix ``$\bs{2}$" that $M$ having integer entries does not necessarily mean that $C$ can be chosen to be a square integer matrix, but it is often still possible to find an integer-$C$ decomposition like $\bs{2}=\bs{\begin{matrix}1&1\end{matrix}}\bs{\begin{matrix}1&1\end{matrix}}^T$ by dropping the requirement that $C$ be square. Which positive-definite matrices $G_+$ have such decompositions? For those that do, can we make use of such decompositions to find a negative-definite matrix $G_-\sim G_+$? We begin by answering the second question, or rather a variant of it:

\begin{theorem}\label{thm: ICCT}
If a positive-definite integer matrix $G{_+}$ takes the form $G{_+}=I+CC^T$ for some integer matrix $C$, then $G{_+}$ is kink-equivalent to a negative-definite matrix, namely ${G_-=}-(I+C^TC)$.
\end{theorem}

\begin{proof}
Let $G{_+}=I+CC^T$, where $C$ is a (not necessarily square) integer matrix. Then:
\[{G{_+}}
\sim\bbme{G_+&0\\0&-I}
\sim\bbme{I&C\\0&I}\bbme{I+CC^T&0\\0&-I}\bbme{I&0\\C^T&I}
=\bbme{I&-C\\-C^T&-I}.\]
Note the symmetry. Doing the same steps in reverse gives
\[\pushQED\qed
\bbme{G{_+}}
\sim\bbme{I&0\\C^T&I}\bbme{I&-C\\-C^T&-I}\bbme{I&C\\0&I}
=\bbme{I&0\\0&-I-C^TC}
\sim
{G_-}.\qedhere\]
\end{proof}

Note that if a matrix $G$ can be represented as $G=I+CC^T$, it can also be represented as {$G=C_0C_0^T$ for $C_0=\bbme{I&C}$}. Hence, if we are to apply Theorem \ref{thm: ICCT}, it is necessary (though not sufficient) that the original matrix be representable as $C_0C_0^T$.
We will see shortly, however, that this is not always possible. First, we 
note this 
 consequence of Theorem \ref{thm: ICCT}:

\begin{cor}\label{C:square}
Every matrix of the form $G=I+CC^T$ for a symmetric integer matrix $C$ is kink-equivalent to $-G$. In particular, this holds for every matrix  $\bs{n^2+1}$, $n\in\Z$.
\end{cor}

{This algebraic fact is not obvious from the motivating topology.   For example, $\bbm 5 \ebm$} is the Gordon-Litherland matrix for a spanning surface for the right hand version of the $5_1$ knot, but not for the left hand version, and vice-versa for $\bbme{-5}$. There is no kink-equivalence between the surfaces as they span distinct knots, but there is a kink-equivalence of matrices:
\begin{align*}
\bs{5}\sim\bbme{1&2\\0&1}\bbme{5&0\\0&-1}\bbme{1&0\\2&1}=&\bbme{1&-2\\-2&-1}\\
=&\bbme{0&-1\\-1&2}\bbme{-5&0\\0&1}\bbme{0&-1\\-1&2}\sim\bs{-5}.
\end{align*}
Further, because the two versions of the $5_1$ knot are not mutants, work of Traldi \cite{traldi} (also see \cite{boninger})
 implies that they share no Goeritz matrix.  Hence, $\bbmes{5&0\\0&-1}$ is a Goeritz matrix only for the right hand version, and $\bbmes{-5&0\\0&1}$ only for the left hand version, even though both are Gordon-Litherland matrices for both knots. (Also see Example \ref{Ex:5}.) We note that it also follows from their geography (again using \cite{ak}) that $\bbmes{-5&0\\0&1}\to [-5]$ is a fake unkinking move for the right hand version of the $5_1$ knot and
 $\bbmes{5&0\\0&-1}\to [5]$ is a fake unkinking move for the left hand version.
We will address the following question in Remark \ref{R:Linking_Forms} (note that kink-equivalence moves preserve the absolute value of the determinant of a matrix):

\begin{question}\label{Q:square}
Is $\bs{3}\sim\bs{-3}$?  More generally, given $g\in\Z_+$ such that $g-1$ is not a perfect square, when is $\bs{g}\sim\bs{-g}$?
\end{question}

Theorem \ref{thm: ICCT} can take us only so far in our pursuit of Question \ref{Q:2}:

 \begin{theorem}\label{thm: counterexample matrix}
The matrix 
\[A=\bbme{2&1&1&1&0&0\\1&2&1&1&1&0\\1&1&2&1&1&1\\1&1&1&2&1&1\\0&1&1&1&2&1\\0&0&1&1&1&2}\] 
is positive-definite and cannot be expressed as $CC^T$ for $C$ an integer matrix.
\end{theorem}

\begin{rem}
The matrix $A$ in Theorem \ref{thm: counterexample matrix} is the Gordon-Litherland matrix for the surface $F\subset S^3$ shown in Figure \ref{Fi:CounterexampleMatrix} with respect to the indicated basis $(a_1,\hdots,a_6)$ for $H_1(F)$.
\end{rem}

\begin{figure}
\centering
\labellist \small
 \pinlabel {$\white{a_1}$} [c] at 118 300
 \pinlabel {${a_2}$} [c] at 512 485
 \pinlabel {$\white{a_3}$} [c] at 410 203
 \pinlabel {${a_4}$} [c] at 330 270
 \pinlabel {${a_5}$} [c] at 565 468
 \pinlabel {$\white{a_6}$} [c] at 567 277
 \pinlabel {$\white{a_1}$} [c] at 45 130
 \pinlabel {${a_2}$} [c] at 160 333
 \pinlabel {$\white{a_3}$} [c] at 210 67
 \pinlabel {${a_4}$} [c] at 193 50
 \pinlabel {${a_5}$} [c] at 590 295
 \pinlabel {$\white{a_6}$} [c] at 490 277
 \pinlabel {${a_2}$} [c] at 370 330
 \pinlabel {$\white{a_3}$} [c] at 162 248
 \pinlabel {${a_4}$} [c] at 160 230
 \pinlabel {${a_5}$} [c] at 437 310
\endlabellist
\includegraphics[width=.6\textwidth]{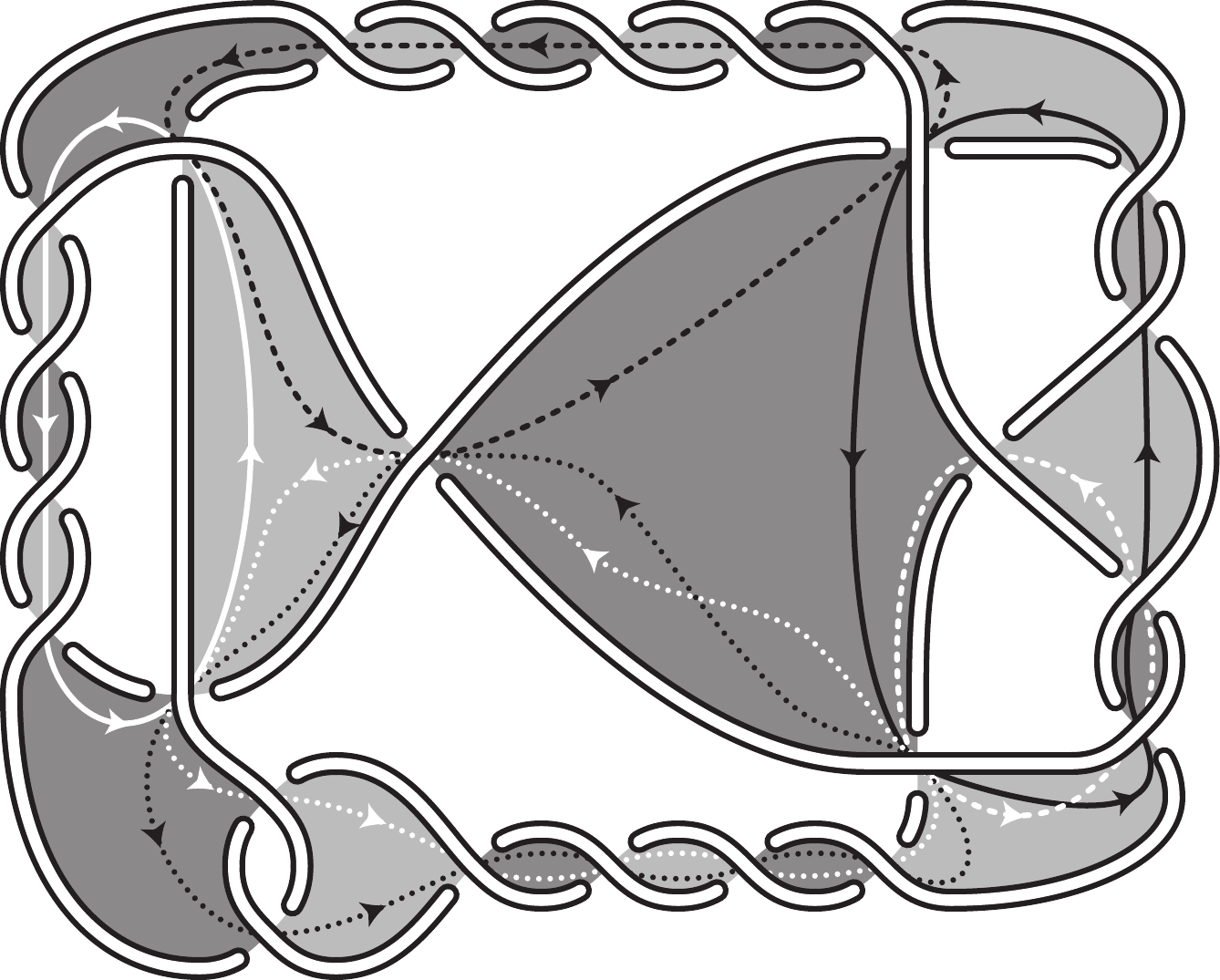}
\caption{Realizing the matrix $A$ from Theorem \ref{thm: counterexample matrix} as a Gordon-Litherland matrix}
\label{Fi:CounterexampleMatrix}
\end{figure}

\begin{proof}[Proof of Theorem \ref{thm: counterexample matrix}]
That this matrix is positive-definite is easily verified directly, so only the second claim requires proof.

Suppose by way of contradiction that there is an integer matrix $C$ such that $CC^T=A$.  Write $C^T=\bs{\begin{matrix}\mathbf{v}_1&\cdots&\mathbf{v}_6\end{matrix}}$, so that $\mathbf{v}_1^T,\hdots,\mathbf{v}_6^T$ are the rows of $C$.  Then $\vb_1^T\vb_1=2$, so we may suppose without loss of generality (by applying signed permutation matrices on the right of $C$, which will cancel with their transposes when multiplying) that $\vb_1^T=\bs{\begin{matrix}1&1&0&0&0&\cdots\end{matrix}}$.  Further, $\vb_2^T\vb_1=1$ and $\vb_2^T\vb_2=2$, so again without loss of generality $\bs{\begin{matrix}\vb_1&\vb_2\end{matrix}}^T=\bbmes{1&1&0&0&0&0&\cdots\\1&0&1&0&0&0&\cdots}$. Continuing with this logic, we have two cases: either 
\[\bbme{\vb_1&\vb_2&\vb_3}^T=\bbme{1&1&0&0&0&0&0&\cdots\\1&0&1&0&0&0&0&\cdots\\0&1&1&0&0&0&0&\cdots}\]
or
\[\bbme{\vb_1&\vb_2&\vb_3}^T=\bbme{1&1&0&0&0&0&0&\cdots\\1&0&1&0&0&0&0&\cdots\\1&0&0&1&0&0&0&\cdots}.\]
In the first case, we have both $\vb_1+\vb_2+\vb_3\equiv\0\text{ (mod 2)}$ and $\vb_4^T\vb_1=\vb_4^T\vb_2=\vb_4^T\vb_3=1$, which is impossible.
In the second case, we need $\vb_4^T\vb_1=\vb_4^T\vb_2=\vb_4^T\vb_3=1$ and $\vb_4^T\vb_4=2$. Without loss of generality, it follows that
\[\bbme{\vb_1&\vb_2&\vb_3&\vb_4}^T=\bbme{1&1&0&0&0&0&0&0&\cdots\\1&0&1&0&0&0&0&0&\cdots\\1&0&0&1&0&0&0&0&\cdots\\1&0&0&0&1&0&0&0&\cdots},\] 
which similarly forces
\[\bbme{\vb_1&\vb_2&\vb_3&\vb_4&\vb_5}^T=\bbme{1&1&0&0&0&0&0&0&\cdots\\1&0&1&0&0&0&0&0&\cdots\\1&0&0&1&0&0&0&0&\cdots\\1&0&0&0&1&0&0&0&\cdots\\1&-1&0&0&0&0&0&0&\cdots}.\] 
This forces $\mathbf{v}_6^T\vb_1\equiv\vb_6^T\vb_5\text{ (mod 2)}$ (regardless of $\vb_6$), and yet to satisfy $CC^T=A$ we need $\mathbf{v}_6^T\vb_1=0$ and $\vb_6^T\vb_5=1$. Contradiction.
\end{proof}

\begin{question}
Which symmetric integer matrices are representable as $CC^T$ for an integer matrix $C$ (not necessarily square)? In particular, 
\begin{enumerate}
\item For what $n$ are 
all $n\times n$ positive-definite matrices are representable as $CC^T$?
\item For what $n$ are 
all $n\times n$ positive-semidefinite matrices are representable as $CC^T$?
\end{enumerate}
\end{question}

 Lagrange's Four Squares Theorem \ref{thm: foursq} and the 
 matrix $A$ from Theorem \ref{thm: counterexample matrix} imply that $2\leq n\leq5$ for both (1) and (2). In fact, for (1), we can improve this a priori lower bound:
 
 \begin{prop}
 Every positive-definite $2\times 2$ integer matrix $A$ is given by $CC^T$ for some  integer matrix $C$ (not necessarily square).
 \end{prop}
 
 \begin{proof}
 By Proposition 5.3.3 of \cite{cohen}, the matrix $A$ is congruent to an integer matrix $A'=\bbmes{a&b\\b&c}$ where $|b|\leq a\leq c$ and, if $a=|b|$ or $a=c$, then $b\geq 0$.  (This matrix $A'$ is called the {\it reduced form} of $A$; see Definition 5.3.2 of \cite{cohen}.) Thus, $A=EA'E^T$ for some unimodular matrix $E$. Let $a'=a-|b|$, $c'=c-|b|$, and
 \[
C= \Big[\underbrace{\bme{1\\0 }~\cdots ~\bme{1\\0 }}_{a'\text{ copies}}\underbrace{\bme{0\\1}~\cdots ~\bme{0\\1}}_{c'\text{ copies}}\underbrace{\bme{1\\ \text{sgn}(b)}~\cdots ~\bme{1\\ \text{sgn}(b)}}_{|b|\text{ copies}}\Big].
 \]
 Then $A'=CC^T$, and thus $A=E(CC^T)E^T=(EC)(EC)^T$.
 \end{proof}

\subsection{Implementing kink-equivalence}

The $I+CC^T$ approach is insufficient for producing kink-equivalences between positive- and negative-definite matrices in general. Indeed, this approach does not work for the matrix $A$ in Theorem \ref{thm: counterexample matrix}. 

For some time, we believed it likely that this matrix $A$ was not kink-equivalent to any negative-definite matrix. Shortly before his Master's thesis was due, however, the fourth author discovered such an equivalence.  Almost immediately, he recognized that the method of that example would generalize completely. This is how Theorem \ref{T:MainMatrix} was discovered.  

The following example shows that this matrix $A$ is indeed kink-equivalent to a negative-definite matrix.  We will then see that this example generalizes, giving the algorithm behind Theorem \ref{T:MainMatrix}. 

\begin{example}\label{Ex:A6}
The matrix $A$ from Theorem \ref{thm: counterexample matrix} is kink-equivalent to the negative-definite matrix $A'=\bbmes{-2&-1\\-1&-2}$ via the following sequence $A=A_0\to A_1\to\cdots\to A_{13}=A'$, in which $A_0\to A_1$ and $A_9\to A_{10}$ are negative kinking moves, $A_3\to A_4$, $A_8\to A_{9}$, and $A_{12}\to A_{13}$ are positive unkinking moves, $A_5\to A_6$ is a triple unkinking move, and all other moves are congruences. We emphasize certain details by coloring them red or blue, and we de-emphasize permutation matrices used for congruence by shading them light gray. A full explanation will come during and after the proof of Theorem \ref{T:MainMatrix}.

The sequence begins with the negative kinking move $A=A_0\to A_1=A_0\oplus\bs{{\color{red}-1}}$ and then uses congruence, followed by a positive unkinking move:
\begin{align*}
A_2&=
\bbme{I&{\color{blue}1}\\0&1}A_1\bbme{I&{\color{blue}1}\\0&1}^T
=\bbmes{{\color{cyan}1}&1&1&1&0&0&{\color{blue}-1}\\1&2&1&1&1&0&0\\1&1&2&1&1&1&0\\1&1&1&2&1&1&0\\0&1&1&1&2&1&0\\0&0&1&1&1&2&0\\{\color{blue}-1}&0&0&0&0&0&{\color{red}-1}} 
=\bbme{{\color{cyan}1}&\vb_2^T\\ \vb_2&M_2},\\
A_3&=
\mlg{\bbme{\0& I\\ 1&\0^T}}\bbme{1&\0^T\\ -\vb_2&I}A_2\bbme{1&\0^T\\ -\vb_2&I}^T\mlg{\bbme{\0& I\\ 1&\0^T}^T}
=\bbme{M_2-\vb_2\vb_2^T&\0 \\ \0^T&{\color{red}1}},\\
A_4&=\bbme{M_2-\vb_2\vb_2^T}=\bbmes{1&0&0&1&0&1\\0&1&0&1&1&1\\0&0&1&1&1&1\\1&1&1&2&1&0\\0&1&1&1&2&0\\1&1&1&0&0&-2}
=\bbme{I&C_6^T\\ C_6&M_6},
\end{align*}
Now a congruence allows three positive unkinking moves, and then two more congruences set up a further positive unkinking move:
\begin{align*}A_5&=
\mlg{\bbme{0&I\\I&0}}\bbme{1&\0^T\\ -C_6&I}A_4\bbme{1&\0^T\\ -C_6&I}^T\mlg{\bbme{0&I\\I&0}^T}
=\bbme{M_6-C_6C_6^T&0\\0&{\color{red}I}}\\
&=\bbmes{-1&-1&-3&\0^T\\ -1&0&-2&\0^T\\ -3&-2&-5&\0^T\\ \0&\0&\0&{\color{red}I}},\\
A_6&=
\bbme{-1&-1&-3\\ -1&0&-2\\ -3&-2&-5},\\
A_7&
=\mlg{\bbme{0&1&0\\1&0&0\\0&0&1}}\bbme{1&0&0\\-1&1&0\\-3&0&1}A_6\bbme{1&0&0\\-1&1&0\\-3&0&1}^T\mlg{\bbme{0&1&0\\1&0&0\\0&0&1}^T}
=\bbme{1&0&1\\0&-1&0\\1&0&4}\\
&=\bbme{{\color{cyan}1}&\vb_7^T\\ \vb_7&M_7},\\
A_8&
=\mlg{\bbme{\0& I\\ 1&\0^T}}\bbme{1&\0^T\\ -\vb_7&I}A_7\bbme{1&\0^T\\ -\vb_7&I}^T\mlg{\bbme{\0& I\\ 1&\0^T}^T}
=\bbme{M_7-\vb_7\vb_7^T&\0 \\ \0^T&{\color{red}1}},\\
A_9&
=\bbme{M_7-\vb_7\vb_7^T}
=\bbme{-1&0\\0&3}.
\end{align*}
We finish as follows:
\begin{align*}
A_{10}&
=\left(\mlg{\bbme{0&1\\1&0}}A_9\mlg{\bbme{0&1\\1&0}^T}\right)\oplus\left[{\color{red}{-1} }\right]=\bbme{{\color{cyan}3}&0&0\\0&{\color{red}-1}&0\\0&0&{\color{red}-1}},\\
A_{11}&
=\bbme{1&{\color{blue}1}&{\color{blue}1}\\0&1&0\\0&0&1}A_{10}\bbme{1&{\color{blue}1}&{\color{blue}1}\\0&1&0\\0&0&1}^T
=\bbme{{\color{cyan}1}&{\color{blue}-1}&{\color{blue}-1}\\{\color{blue}-1}&{\color{red}-1}&0\\{\color{blue}-1}&0&{\color{red}-1}}
=\bbme{1&\vb_{11}^T\\ \vb_{11}&M_{11}},\\
A_{12}&
=\mlg{\bbme{0&1&0\\0&0&1\\1&0&0}}\bbme{1&\0^T\\ -\vb_{11}&I}A_{11}\bbme{1&\0^T\\ -\vb_{11}&I}^T\mlg{\bbme{0&1&0\\0&0&1\\1&0&0}^T}
=A'\oplus\bs{{\color{red}1}},\\
A_{13}&=A'.
\end{align*}
\end{example}

\begin{T:MainMatrix}
Every kink-equivalence class of symmetric integer matrices contains positive- and negative-semidefinite representatives. 
Moreover,  given a nonsingular symmetric matrix $G\in\Z^{n\times n}$ with $n_+$ positive eigenvalues and $n_-$ negative eigenvalues, there are $\pm$-definite 
integer 
matrices $A_\pm$ that satisfy 
$G\oplus I_{4n_-}\cong A_+\oplus I_{n_-}$ 
and $G\oplus-I_{4n_+}\cong A_-\oplus I_{n_+}$.
\end{T:MainMatrix}

\begin{proof}
Let $G\in\Z^{n\times n}$ be symmetric with $n_+$ positive eigenvalues, $n_+>0$.  We will describe an algorithm which gives a sequence of kink-equivalence operations taking $G$ to a symmetric integer matrix with $n_+-1$ positive eigenvalues.  This sequence will involve at most four negative kinking moves and no positive ones, exactly one positive  unkinking move and no negative ones, and several congruences.  The full theorem will then follow by symmetry (between the behavior of positive versus negative eigenvalues) and induction (on $n_+$). 

The algorithm has three steps, the first of which will warrant a closer look.  

{\bf Step 1:} Find a congruence $G=G_0\to G_1=P_0G_0P_0^T$ so that $G_1$ has a positive entry in its top-left corner, say $G_1=\bbmes{k&\vb^T\\ \vb&M}$, where $k\in\Z_+$, $\vb\in\Z^{n-1}$, and $M\in\Z^{(n-1)\times(n-1)}$.  In Example \ref{Ex:A6}, this leaves the given matrix $A$ unchanged.

{\bf Step 2:} Using the Four Squares Theorem, find $a,b,c,d\in\Z$ with $a^2+b^2+c^2+d^2=k-1$. 
(There are practical algorithms available for finding particular values for $a,b,c,d$ given $k$: see for instance \cite{pt18}.) Let $G_2=G_1\oplus {-I_4}$; if some of  $a,b,c,d$ are zero, one has the option of adding fewer than four (negative) kinks here. Then apply the congruence $G_2\to G_3$ where
\begin{align*}
G_3=&\bs{\begin{smallmatrix}1&\Lg{\zeros^T}&{\color{blue}a}&{\color{blue}b}&{\color{blue}c}&{\color{blue}d}\\\Lg{\zeros}&I&\Lg{\zeros}&\Lg{\zeros}&\Lg{\zeros}&\Lg{\zeros}\\\Lg{0}&\Lg{\zeros^T}&1&\Lg{0}&\Lg{0}&\Lg{0}\\\Lg{0}&\Lg{\zeros^T}&\Lg{0}&1&\Lg{0}&\Lg{0}\\\Lg{0}&\Lg{\zeros^T}&\Lg{0}&\Lg{0}&1&\Lg{0}\\\Lg{0}&\Lg{\zeros^T}&\Lg{0}&\Lg{0}&\Lg{0}&1\end{smallmatrix}}\bs{\begin{smallmatrix}k&\mathbf{v}^T&\Lg{0}&\Lg{0}&\Lg{0}&\Lg{0}\\\mathbf{v}&M&\Lg{\zeros}&\Lg{\zeros}&\Lg{\zeros}&\Lg{\zeros}\\\Lg{0}&\Lg{\zeros^T}&{\color{red}-1}&\Lg{0}&\Lg{0}&\Lg{0}\\\Lg{0}&\Lg{\zeros^T}&\Lg{0}&{\color{red}-1}&\Lg{0}&\Lg{0}\\\Lg{0}&\Lg{\zeros^T}&\Lg{0}&\Lg{0}&{\color{red}-1}&\Lg{0}\\\Lg{0}&\Lg{\zeros^T}&\Lg{0}&\Lg{0}&\Lg{0}&{\color{red}-1}\end{smallmatrix}}\bs{\begin{smallmatrix}1&\Lg{\zeros^T}&{\color{blue}a}&{\color{blue}b}&{\color{blue}c}&{\color{blue}d}\\\Lg{\zeros}&I&\Lg{\zeros}&\Lg{\zeros}&\Lg{\zeros}&\Lg{\zeros}\\\Lg{0}&\Lg{\zeros^T}&1&\Lg{0}&\Lg{0}&\Lg{0}\\\Lg{0}&\Lg{\zeros^T}&\Lg{0}&1&\Lg{0}&\Lg{0}\\\Lg{0}&\Lg{\zeros^T}&\Lg{0}&\Lg{0}&1&\Lg{0}\\\Lg{0}&\Lg{\zeros^T}&\Lg{0}&\Lg{0}&\Lg{0}&1\end{smallmatrix}}^T\\
=&\bs{\begin{smallmatrix}{\color{cyan}1}&\mathbf{v}^T&{\color{blue}-a}&{\color{blue}-b}&{\color{blue}-c}&{\color{blue}-d}\\\mathbf{v}&M&\Lg{\zeros}&\Lg{\zeros}&\Lg{\zeros}&\Lg{\zeros}\\{\color{blue}-a}&\Lg{\zeros^T}&{\color{red}-1}&\Lg{0}&\Lg{0}&\Lg{0}\\{\color{blue}-b}&\Lg{\zeros^T}&\Lg{0}&{\color{red}-1}&\Lg{0}&\Lg{0}\\{\color{blue}-c}&\Lg{\zeros^T}&\Lg{0}&\Lg{0}&{\color{red}-1}&\Lg{0}\\{\color{blue}-d}&\Lg{\zeros^T}&\Lg{0}&\Lg{0}&\Lg{0}&{\color{red}-1}\end{smallmatrix}}.
\end{align*}
The $1$ in the top-left of $G_3$ comes from $k-a^2-b^2-c^2-d^2=k-\ps{k-1}=1$. In Example \ref{Ex:A6}, Step 2 takes $A_0\to A_1\to A_2$.

{\bf Step 3:} Write $G_3=\bbmes{1&\wb^T\\ \wb&N}$, where $\wb^T=\bbme{\vb^T&-a&-b&-c&-d}$ and $N=M\oplus-I_4$. Use the 1 in the top-left to eliminate the rest of its row and column by the congruence
\begin{align*}
G_3\to G_4=&\Lg{\bbme{\0&I\\1&\0^T}}\bbme{1&\0^T\\-\wb&I}\bbme{\fbox{1}&\wb^T\\ \wb&N}\bbme{1&\0^T\\-\wb&I}^T\Lg{\bbme{\0&I\\1&\0^T}}\\
=&\left(N-\wb\wb^T\right)\oplus\bs{1}.
\end{align*}
Perform the positive  unkinking move $G_4=\left(N-\wb\wb^T\right)\oplus\bs{1}\to N-\wb\wb^T=G_5$.  By Sylvester's Inertia Theorem \ref{T:Syl}, $G_0,\hdots,G_4$ all have exactly $n_+$ positive eigenvalues and $G_5$ has $n_+-1$ of them. In Example \ref{Ex:A6}, Step 3 takes $A_2\to A_3\to A_4$. 

Why is Step 1 always possible? The function $f:\R^n\to\R$ given by $f:\mathbf{x}\mapsto\mathbf{x}^TG\mathbf{x}$ attains a positive value because $G$ has a positive eigenvalue, and there is a rational vector $\mathbf{u}\in\Q^n$ with $\mathbf{u}^TG\mathbf{u}>0$ because $\Q^n$ is dense in $\R^n$ and $f$ is continuous. %
\footnote{For the sake of keeping the algorithm as constructive as possible: it is also possible to obtain such a $\mathbf{u}$ by first adding a large multiple of the identity to $A$ to produce a matrix with no negative eigenvalues (Gersgorin disks can be used to efficiently find a sufficiently high multiple of $I$ to add) and then applying power iteration to the resulting matrix (modified to keep the normalizing factors rational) to find an explicit computable sequence of rational vectors converging to the eigenvector for $A$ with largest (necessarily positive) eigenvalue. Some entry in this sequence will be an acceptable choice of ${\mathbf{u}}$.}
 Write the entries of $\mathbf{u}$ in lowest terms, let $d_1$ be the least common multiple of their denominators, and  let $d_2$ be the least common multiple of their numerators.  Then $\mathbf{b}=\frac{d_1}{d_2}{\mathbf{u}}$ 
is a primitive integer vector satisfying $\mathbf{b}^TG\mathbf{b}>0$, and so by Fact \ref{F:Primitive} there is a unimodular matrix $P$ whose first column is $\mathbf{b}$.  Finally, denoting the first standard basis vector for $\Z^n$ by $\mathbf{e}_1$, the top-left entry in $P^TGP$ is $\mathbf{e}_1^TP^TGP\mathbf{e}_1=\mathbf{b}^TG\mathbf{b}>0$, so $G\to P^TGP$ is our desired congruence for Step 1.
\end{proof}

\section{Implications and further questions}\label{S:Final}

\subsection{Kink-equivalence to one-by-one matrices}

It is interesting to observe that the matrices $A$ and $A'=\bbmes{-2&-1\\-1&-2}$ in Example \ref{Ex:A6} are  kink-equivalent to a much simpler positive-definite matrix than $A$:
\begin{align*}
A'&\sim\bbme{-2&-1&0\\-1&-2&0\\0&0&1}\sim\bbme{1&0&1\\0&1&0\\0&0&1}\bbme{-2&-1&0\\-1&-2&0\\0&0&1}\bbme{1&0&1\\0&1&0\\0&0&1}^T\\
&=\bbme{-1&-1&1\\-1&-2&0\\1&0&1}\sim\bbme{2&-1&1\\-1&1&0\\1&0&0}\bbme{-1&-1&1\\-1&-2&0\\1&0&1}\bbme{2&-1&1\\-1&1&0\\1&0&0}^T\\
&=\bbme{3&0&0\\0&-1&0\\0&0&-1}\sim\bs{3}.
\end{align*}
This motivates the following question, which is related to Question \ref{Q:square}:

\begin{question}\label{Q:1x1}
Which kink-equivalence classes contain one-by-one matrices?
\end{question}

\subsection{Consequences} While Gordon-Litherland matrices for spanning surfaces provided our initial motivation, Theorem \ref{T:MainMatrix} has implications in any mathematical setting where symmetric integer pairings are used. In particular, Theorem \ref{T:MainMatrix} has consequences not just for knots, links, and spanning surfaces, but also for 4-manifolds and quadratic forms.  

\subsubsection{Knots, links, and spanning surfaces}

Recall Greene's Theorem \ref{T:Greene} that a {link }
is alternating if and only if it has spanning surfaces $F_+$ and $F_-$ whose respective Gordon-Litherland matrices are positive- and negative-definite \cite{greene}.
{In light of Greene's characterization and Theorem \ref{T:encyc}, it is natural to ask if the alternating condition of a link is encoded algebraically in the kink-equivalence class that contains the Gordon-Litherland matrices of that link's checkerboard surfaces. That is, how much information does one lose by ``naively'' considering kink-equivalence of Gordon-Litherland matrices rather than the more restrictive kink-equivalence of spanning surfaces?  

A lot, it turns out.  Theorem \ref{T:MainMatrix} implies that the Gordon-Litherland pairings for the checkerboard surfaces for any knot belong to a kink-equivalence class of matrices which contains positive- and negative-definite matrices. Thus, the question of whether or not a given knot is alternating balances not on the kink-equivalence class of matrices containing its surfaces' Gordon-Litherland pairings, but rather on which matrices in this class are actually realized by surfaces. That is, the alternating condition is determined at the level of topology, not pure linear algebra. On the contrary, every knot is in some sense ``algebraically alternating'':

\begin{theorem}\label{T:MainKnotLink}
Given a link $L\subset S^3$ with nullity 0 and a Gordon-Litherland matrix $G$ for any spanning surface for $L$, $G$ is kink-equivalent to a positive-definite matrix and to a negative-definite matrix. In particular, this is true for every {\it knot} in $S^3$.

{Moreover, any Gordon-Litherland matrix for any spanning surface for any link in $S^3$ is kink-equivalent to positive- and negative-semidefinite matrices.}
\end{theorem}

\begin{proof}
    The matrices in every kink-equivalence class have uniform nullity, so if this nullity is 0, then Theorem \ref{T:MainMatrix} implies that the class contains positive- and negative-definite matrices. In particular, this is true for knots, which always have odd determinant and thus nullity 0. The last claim follows immediately from the first sentence of Theorem \ref{T:MainMatrix}.
\end{proof}

\subsubsection{4-manifolds}\label{S:4manifolds}
The ideas in this paper, which we have presented in terms of Gordon-Litherland pairings on (unoriented) spanning surfaces and how they change under positive and negative kinking moves, also apply to intersection pairings on (orientable) 4-manifolds and how they change under blowups and anti-blowups, i.e. connected summing with $\overline{\mathbb{CP}^2}$'s and $\mathbb{CP}^2$'s.  In particular, Theorem \ref{T:MainMatrix} can also be interpreted in terms of intersection forms on 4-manifolds vis-a-vis stabilization under blowup and anti-blowup.  See Theorem \ref{T:Main4Manifold1}.

\begin{proof}[Proof of Theorem \ref{T:Main4Manifold1}]
Write $Q$ for a matrix representing the intersection pairing on $H_2(M)$, and write $n_\pm$ for the number of its eigenvalues of $Q$ of either sign. Theorem \ref{T:MainMatrix} implies that there are $\pm$-definite 
integer matrices $A_\pm=[a^\pm_{ij}]$ that satisfy the congruences $Q\oplus I_{4n_\mp}\cong A_\pm\oplus -I_{n_\mp}$. Let $L_\pm$ be oriented links in $S^3$ with components $L_i^\pm$ with linking numbers $\text{lk}(L_i^\pm,L_j^\pm)=a^\pm_{ij}$. Construct $M_\pm$ by attaching a 2-handle with framing $a^\pm_{ii}$ along each $L^\pm_i$. Each $M_\pm$ is smooth and simply connected with intersection pairing represented by $A_\pm$.
\end{proof} 

\begin{example}\label{Ex:5}
The proof of Theorem \ref{T:Main4Manifold1} also suggests a simple way to construct a sequence of 4-manifolds realizing the kink equivalence $[5]\sim[-5]$.  Namely, take $M_\pm$ to be the trace of the $\pm5$-framed unknot; each intersection pairing is represented by the matrix $[\pm 5]$.  Then the intersection forms for $M_+\#\overline{\mathbb{CP}^2}$ and $M_-\#\mathbb{CP}^2$ are represented by the congruent matrices $\bbmes{5&0\\0&-1}$ and $\bbmes{-5&0\\0&1}$.
\end{example}

\begin{question}\label{q:4mfld}
Can we always choose $M_\pm$ in the conclusion of Theorem \ref{T:Main4Manifold1} such that $M\#^{4n_+}\overline{\CP^2}$ and $M\#^{4n_-}{\CP^2}$ are respectively  {\it homeomorphic} to $M_-\#^{n_+}\CP^2$ and $M_+\#^{n_-}\overline{\CP^2}$? If $M$ is smooth, when can we choose $M_{\pm}$ such that these stabilizations are {\it diffeomorphic?}
\end{question}

\begin{rem}\label{rem:fake}
The notion of ``fake unkinking'' captured in Example \ref{Ex:Naive} also finds a direct parallel in fake blowdowns and fake anti-blowdowns of smooth 4-manifolds.  For certain smooth 4-manifolds, it appears possible, on the level of their intersection forms, to de-sum a ${\mathbb{CP}^2}$ or a $\overline{\mathbb{CP}^2}$, even though this operation is not possible. 

Perhaps the most famous such examples are the Dolgachev surfaces $E(1)_{p,q}$, where $p,q\geq 2$ are relatively prime, which are obtained from the rational elliptic surface $E(1)=\mathbb{CP}^2\#^9\overline{\mathbb{CP}^2}$ by a pair of logarithmic transformations \cite{dolgachev}. Each Dolgachev surface $E(1)_{p,q}$ is homeomorphic to $E(1)$, having intersection form $[1]\oplus-I_9$, and appears, in this sense, to admit many blowdowns and an anti-blowdown. Yet, by work of Donaldson and Freedman-Morgan, no $E(1)_{p,q}$ actually admits a blowdown, and so each $E(1)_{p,q}$ is exotic, i.e. not diffeomorphic to $E(1)$ \cite{donaldson1,donaldson2,fm1}.

Further, the anti-blowup $E(1)_{p,q}\#\mathbb{CP}^2$ of any Dolgachev surface is diffeomorphic to $E(1)\#\mathbb{CP}^2$ \cite{mm}. %
\footnote{The Dolgachev surfaces are thus said to be {\it almost completely decomposable.}}  Hence, taking $M=E(1)_{p,q}$ in Theorem \ref{T:Main4Manifold1}, so that $n_+=1$ and $n_-=9$, and taking $M_+=\#^{37}\mathbb{CP}^2$, we have diffeomorphisms 
\[M\#^{4n_-}\mathbb{CP}^2\cong  E(1)\#^{36}\mathbb{CP}^2\cong\#^{37}\mathbb{CP}^2\#^{9}\overline{\mathbb{CP}^2}\cong M_+\#^{n-}\overline{\mathbb{CP}^2}.\]
This gives an example in which at least one of the diffeomorphisms sought in Question \ref{q:4mfld} is known to exist.  Moreover, in this same example, if there is a negative-definite manifold $M_-$ for which there is a diffeomorphism
\[M\#^{4n_+}\overline{\mathbb{CP}^2}=M\#^{4}\overline{\mathbb{CP}^2}\cong M^-\#\mathbb{CP}^2=M^-\#^{n_+}\mathbb{CP}^2,\]
then $M_-$ must be homeomorphic to $\#^{13}\overline{\mathbb{CP}^2}$.  Yet, such $M_-$ could not be diffeomorphic to $\#^{13}\overline{\mathbb{CP}^2}$ because $E(1)_{p,q}$ and $E(1)$ remain smoothly distinct even after an arbitrary number of blowups \cite{fs}.  Thus, if such $M_-$ exists in this example, it must be a negative-definite, simply connected, closed 4-manifold.  In general, the existence of definite, simply connected, closed 4-manifolds remains an important open question.
\end{rem}

\begin{rem}\label{R:Linking_Forms}
    The matrices $A_\pm$ in the proof of Theorem \ref{T:Main4Manifold1} give rise to {\it linking forms} on the 3-manifolds $\partial M_\pm$---these are symmetric bilinear maps on the torsion part of $H_1(\partial M_\pm)$ with outputs in $\mathbb{Q}/\mathbb{Z}$. In fact, the construction, which is purely algebraic, works for any linking matrix, definite or not, or indeed on any closed 3-manifold (see e.g. \cite{cfh,m15}).  Rather than surveying the extensive literature on linking forms, we highlight a few key points.  
    \begin{itemize}
        \item In \cite[Proposition 2.4]{moo}, summarizing results of Kneser--Puppe, Wall, and Durfee \cite{kp,wall,durfee}, Murakami--Ohtsuki--Okada describe how kink-equivalence of linking matrices $A$ (which is called ``stable equivalence'' in the literature on linking forms) is equivalent to isomorphism of linking forms.
        \item Wall and Kawauchi--Kojima \cite{wall,kk} give a complete algebraic classification of linking forms.
        \item The perspective of linking forms is especially illuminating with regard to Question \ref{Q:square} and Example \ref{Ex:5}. In short, the fact that $-1$ is a quadratic residue modulo 5 implies that $[5]$ and $[-5]$ determine the same linking form and therefore are kink-equivalent, whereas $[3]$ and $[-3]$ determine distinct linking forms and therefore cannot be kink-equivalent.  This answers at least the first part of Question \ref{Q:square}.  The classical theory of linking forms might also provide an answer to Question \ref{Q:1x1}
        \item We found no analysis in the prior literature of kink-equivalence to definite matrices (let alone with bounds on numbers of stabilization moves), so our main results seem to be new.  It would be interesting to see what light our results might shed on the classical theory of linking forms.
    \end{itemize}
\end{rem}

\begin{rem}\label{R:Stabilization}
We choose to use our new terminology of ``kink-equivalence'' rather than the classical terminology of ``stable equivalence'' in part because the latter is used in a variety of settings, sometimes in closely related but distinct ways.  In particular, ``stable equivalence'' of 4-manifolds often refers to Wall's theorem that, given closed 4-manifolds $M$ and $M'$, with isomorphic intersection forms, $M\#^n(S^2\times S^2)$ is diffeomorphic to $M'\#^n(S^2\times S^2)$ for sufficiently large $n$ \cite{wallstab}. In this context, ``stable equivalence'' of symmetric integer matrices might naturally refer to their equivalence classes under unimodular congruence and $G\leftrightarrow\bbmes{0&1\\1&0}$, which differs only slightly from the notion under consideration in this paper.  

Likewise, in the context of a Witt ring $W(R)$, one
considers (matrices representing) inner product spaces over a commutative ring $R$ with unity \cite{mh}. One regards two elements $X$ and $X'$ in the Witt ring as equal if $X \oplus S$ is congruent to $X'\oplus S'$ for some {\it split} inner product spaces $S$ and $S'$. In particular, in $W(R)$, one can take away summands of the form $\bbmes{ 1&0\\0&-1}$, which is split, but not simply $[\pm1]$ as one can with kink-equivalence.
\end{rem}

\subsubsection{Quadratic forms}

Our results may also be interpreted  in terms of  quadratic forms.   Theorem \ref{T:MainMatrix} implies that, up to (unimodular) change of basis and a stability condition generated by the addition or removal of variables as $Q(x_1,\hdots,x_n)\leftrightarrow Q(x_1,\hdots,x_n)\pm x_{n+1}^2$, any quadratic form $Q:\Z^n\ar\Z$ that comes from a symmetric integer matrix $A\in \Z^{n\times n}$ (as $Q:x\mapsto\mathbf{x}^TA\mathbf{x}$) appears both positive- and negative-definite. That is,  writing the quadratic form $q_0:x\mapsto x^2$, we have the following:

\begin{theorem}\label{T:MainQuadratic}
Let $q:\Z^n\to \Z$ be a nonsingular quadratic form with even-coefficient cross-terms, and write $n_\pm=\frac{1}{2}(n\pm\sigma(q))$. There are $\pm$-definite quadratic forms $q_\pm$ that satisfy the congruences $q\oplus (q_0)^{  4n_-}\cong q_+\oplus (-q_0)^{n_-}$ and $q\oplus (-q_0)^{  4n_+}\cong q_-\oplus (q_0)^{n_+}$.
\end{theorem}

 Can we extend Theorem \ref{T:MainMatrix} to matrices with half-integers off the diagonal and thus remove the cross-term condition from Theorem \ref{T:MainQuadratic}? 
 Indeed, we can. This is next.
 
 \section{Extension to symmetric rational matrices}\label{S:Quadratic}

We find that Theorem \ref{T:MainMatrix} extends to the case of rational matrices, at the cost of increasing the bound on the number of stabilizations required from $4n$ to $5n$. In extending kink-equivalence to rational matrices we do not relax the requirement that congruence moves are by (integer) unimodular matrices.

\begin{theorem}\label{T:RationalExtension}
Every kink-equivalence class of symmetric rational matrices contains positive- and negative-semidefinite representatives. Moreover, given a nonsingular symmetric matrix $G\in\Q^{n\times n}$ with $n_+$ positive eigenvalues and $n_-$ negative eigenvalues, there are $\pm$-definite 
rational 
matrices $A_\pm$ that satisfy the unimodular congruences  $G\oplus I_{5n_-}\cong A_+\oplus -I_{n_-}$ and $G\oplus-I_{5n_+}\cong A_-\oplus I_{n_+}$.%
\end{theorem}

\begin{proof}
Step 1 of the algorithm remains unchanged: apply a unimodular congruence to produce a positive entry in the top-left corner of the matrix. Steps 2 and onward will also remain unchanged, but an integralization step must be inserted between Steps 1 and 2 to ensure that the algorithm continues to use unimodular congruences rather than congruences by general rational unit-determinant matrices. 

{\bf Integralization step:} The output of Step 1 is a symmetric rational matrix with positive top-left entry. Call the least common multiple of the denominators in the first row of this matrix $d$, so that it can be written as $\bbmes{k/d&&\mathbf{v}^T\!/{d}\\ \mathbf{v}/{d}&&M}$ with $M$ rational,  $k$ a positive integer, and $\mathbf{v}$ a vector with integer entries. If $d=1$ we can proceed to Step 2 immediately. Otherwise, apply a negative kinking move and then a congruence as follows:
\begin{align*}
\bbm\frac{k}{d}&\frac{\mathbf{v}^T}{d}\\\frac{\mathbf{v}}{d}&M\ebm
\sim&\bbm\frac{k}{d}&\frac{\mathbf{v}^T}{d}&0\\\frac{\mathbf{v}}{d}&M&\zeros\\0&\zeros^T&-1\ebm
\\
\sim&\bbm d&\zeros^T&1\\\zeros&I&\zeros\\d-1&\zeros^T&1\ebm\bbm\frac{k}{d}&\frac{\mathbf{v}^T}{d}&0\\\frac{\mathbf{v}}{d}&M&\zeros\\0&\zeros^T&-1\ebm\bbm d&\zeros^T&1\\\zeros&I&\zeros\\d-1&\zeros^T&1\ebm^T\\
=&\bbm dk-1&\mathbf{v}^T&\ps{d-1}k-1\\\mathbf{v}&M&\frac{d-1}{d}\mathbf{v}\\\ps{d-1}k-1&\frac{d-1}{d}\mathbf{v}^T&\frac{\ps{d-1}^2}{d}k-1\ebm.
\end{align*}
The resulting matrix has strictly positive top-left entry, as $k\geq 1$ and $d>1$, and its first row and column are integer vectors. It is therefore suitable for insertion into Step 2. 
\end{proof}

Theorem \ref{T:RationalExtension} allows us to extend Theorem \ref{T:MainQuadratic} to all integer quadratic forms, without the condition about even cross-terms, and in fact to all rational quadratic forms, at the cost of slightly weakening the bound on the number of stabilizations required.  Again, when extending kink-equivalence to such quadratic forms, we do not relax the requirement that the congruence moves are unimodular. 

\begin{T:MainQuadratic2}
Let $q:\Z^n\to \Q$ be a nonsingular quadratic form, and write $n_\pm=\frac{1}{2}(n\pm\sigma(q))$ and $q_0:x\mapsto x^2$. There are $\pm$-definite quadratic forms $q_\pm$ that satisfy the congruences $q\oplus (q_0)^{  5n_-}\cong q_+\oplus (-q_0)^{n_-}$ and $q\oplus (-q_0)^{  5n_+}\cong q_-\oplus (q_0)^{n_+}$.
\end{T:MainQuadratic2}

{\bf Acknowledgements:} Thank you to the anonymous referee for detailed feedback which greatly improved the paper, and thank you to Paul Melvin and Brendan Owens for helpful conversations.


\begin{thebibliography}{99}



 \bibitem
{ak} C. Adams and T. Kindred, {A classification of spanning surfaces for alternating links}, Alg. Geom. Topology 13 (2013), no. 5, 2967--3007.
%
 \bibitem
{bnfk} D. Bar-Natan and J. Fulman and L.H. Kauffman, {An elementary proof that all spanning surfaces of a link are tube-equivalent}, J. Knot Theory Ramifications 7 (1998), no. 7, 873--879. 
%
\bibitem
{boninger} J. Boninger, {The Jones polynomial from a Goeritz matrix}, Bull. Lond. Math. Soc. 55 (2023), no. 2, 732--755.
%
\bibitem
{cohen} H. Cohen, {A course in computational algebraic number theory}, Grad. Texts Math., 138, {Springer-Verlag, Berlin}, (1993), xii+534 pp.
%
\bibitem
{cfh} A. Conway and S. Friedl and G. Herrmann, {Linking forms revisited}, Pure Appl. Math. Q. 12 (2018), no. 4, 493--515.
%
\bibitem
{coward} A. Coward, {Ordering the Reidemeister moves of a classical knot}, 
Alg. Geom. Topology 6 (2006), no. 2, 659–667.
%
\bibitem
{dolgachev}
I. Dolgachev, {Algebraic surfaces with $q=p_q=0$}, Algebraic Surfaces, Liguori,
Napoli, 1981, pp. 97--215.
%
\bibitem
{donaldson1}
S.K. Donaldson, {La topologie diff\'erentielle de surfaces complexes}, CR Acad. Sci. Paris
(I) 301 (1985), 317--320.
%
\bibitem
{donaldson2} S.K. Donaldson, {Irrationality and the $h$-cobordism conjecture}, J. Dff. Geom. 26 (1987), 141--168. 
%
\bibitem
{durfee} A.H. Durfee, {Bilinear and quadratic forms on torsion modules}, Adv. in Math.
25(1977), 133--164.
%
\bibitem
{fs} R. Fintushel and R. Stern, {The blowup formula for Donaldson invariants}, Ann. of Math. 143 (1996), 529–546.
%
\bibitem
{freedman} M. Freedman, {The topology of four-dimensional manifolds}, J. Diff. Geom. 17 (1982), 357--453.
%
\bibitem
{fm1} R. Freedman and J. Morgan, {On the diffeomorphism types of certain algebraic surfaces I}, J. Diff. Geom. 27 (1988), 297–369; 
%
\bibitem
{gs} R.E. Gompf and A.I. Stipsicz, {4-manifolds and Kirby calculus}, Grad. Stud. Math., 20, American Mathematical Society, Providence, RI, (1999), xvi+558 pp.
%
 \bibitem
{gl} C. McA. Gordon and R.A. Litherland, {On the signature of a link}, Invent. Math. 47 (1978), no. 1, 53--69.
%
\bibitem
{greene} J. Greene, {Alternating links and definite surfaces}, Duke Math. J. 166 (2017), no. 11, 2133--2151. 
%
\bibitem
{kt} L.H. Kauffman and L. Taylor, {Signature of links}, Trans. Amer. Math. Soc. 216 (1976), 351--365.
%
\bibitem
{kk} A. Kawauchi and S. Kojima, {Algebraic classification of linking pairings on 3-manifolds}, Math. Ann. 253 (1980), no. 1, 29--42.
%
\bibitem
{thesis} T. Kindred, {Checkerboard plumbings}, Ph.D. thesis, University of Iowa (2018).
%
 \bibitem
{encyc} T. Kindred, {Nonorientable spanning surfaces for knots}, Chapter 23 of the Concise Encyclopedia of Knot Theory (2021), 197--203.
%
\bibitem
{essence} T. Kindred, {How essential is a spanning surface?}, arXiv:2408.16948.
%
\bibitem
{kp} M. Kneser and D. Puppe, {Quadratische Formen und Verschlingungsinvarianten von Knoten} (in German),
Math. Z. 58 (1953), 376--384
%
\bibitem
{kob} M. Kobayashi and T. Kobayashi, {On canonical genus and free genus of knot}, J. Knot Theory Ramifications 5 (1996), no. 1, 77--85.
%
\bibitem
{levine} J. Levine, {Knot cobordism groups in codimension two}, Comment. Math. Helv. 44 (1969), 229--244.
 %
 \bibitem
 {mm} R. Mandelbaum and B. Moishezon, {On the topology of simply-connected algebraic surfaces}, Trans. Amer. Math. Soc. 260 (1980) no. 1, 197--222.
 %
 \bibitem
{mh} J.W. Milnor and D. Husem\"oller, {Symmetric bilinear forms}, Springer, 1973.
%
\bibitem
{m15} D. Moussard, {Realizing isomorphisms between first homology groups of closed 3-manifolds by Borromean surgeries}, J. Knot Theory Ramifications 24 (2015) no. 4, 17 pp.
%
\bibitem
{moo} H. Murakami and T. Ohtsuki and M. Okada, {Invariants of three-manifolds derived from linking matrices of framed links}, Osaka J. Math. 29 (1992), no. 3, 545–572.
 %
  \bibitem
 {mur65} K. Murasugi, {On a certain numerical invariant of link types}, Trans. Amer. Math. Soc. 117 (1965), 387--422.
 %
 \bibitem
{pt18} P. Pollack and E. Trevi\~no, {Finding the four squares in Lagrange’s theorem}, Integers 18A (2018), Paper No. A15, 16pp.
 %
 \bibitem
{rado} R. Rado, {A proof of the basis theorem for finitely generated Abelian groups}, J. London Math. Soc. 26 (1951), 74--75.
 %
 \bibitem
{serre} J.P. Serre, {A course in arithmetic}, Grad. Texts in Math., no. 7, Springer-Verlag, New York-Heidelberg, 1973, viii+115 pp.
%
\bibitem
{traldi} L. Traldi, {Link mutations and Goeritz matrices}, Osaka J. Math. 59 (2022), no. 4, 881--904.
%
\bibitem
{wall} C.T.C. Wall, {Quadratic forms on finite groups, and related topics}, Topology 2 (1963), 281--298.
%
\bibitem
{wallstab} C.T.C. Wall, {On simply connected 4-manifolds}, J. London Math. Soc. 39 (1964), 141--149.
%
\bibitem
{yas} A. Yasuhara, {An elementary proof that all unoriented spanning surfaces of a link are related by attaching/deleting tubes and Mobius bands}, J. Knot Theory Ramifications 23 (2014), no. 1, 5 pp.
\end{thebibliography}
\end{document}